\documentclass{article}

\usepackage{amsmath}
\usepackage{amssymb}
\usepackage{mathabx}
\usepackage{amsthm}
\usepackage{color}
\usepackage{enumerate}
\usepackage{relsize}
\usepackage{multicol}

\newtheorem{theorem}{Theorem}[section]
\newtheorem{lemma}[theorem]{Lemma}
\newtheorem{proposition}[theorem]{Proposition}
\newtheorem{corollary}[theorem]{Corollary}

\theoremstyle{definition}

\newtheorem{example}[theorem]{Example}

\theoremstyle{nonumberplain}

\newcommand{\imp}{\mathbin{\rightarrow}}

\newcommand{\eq}{\mathrel{\sim}}
\newcommand{\rad}{\mathop{\mathrm{rad}}}
\newcommand{\radn}{\mathop{\mathbf{rad}}}
\newcommand{\Mod}{\mathop{\mathrm{Mod}}}

\usepackage{bm}
\newcommand{\bigand}{\boldsymbol{\bigwedge}\displaylimits}
\newcommand{\band}{\boldsymbol{\land}}
\newcommand{\bor}{\boldsymbol{\lor}}
\newcommand{\bimp}{\mathbin{\boldsymbol{\rightarrow}}}
\newcommand{\biff}{\mathbin{\boldsymbol{\leftrightarrow}}}

\begin{document}

\title{Algebraic expansions of logics and algebras and a case study of Abelian $\ell$-groups and perfect MV-algebras}

\author{M. Campercholi, D. Casta\~no, J. P. D\'{\i}az Varela, J. Gispert}

\maketitle

\begin{abstract}
An algebraically expandable (AE) class is a class of algebraic structures axiomatizable by sentences of the form $\forall \exists! \bigand p = q$. For a logic $L$ algebraized by a quasivariety $\mathcal{Q}$ we show that the AE-subclasses of $\mathcal{Q}$ correspond to certain natural expansions of $L$, which we call {\em algebraic expansions}. These turn out to be a special case of the expansions by implicit connectives studied by X. Caicedo. We proceed to characterize all the AE-subclasses of Abelian $\ell$-groups and perfect MV-algebras, thus fully describing the algebraic expansions of their associated logics.
\end{abstract}

\section{Introduction}

The idea of expanding structures in a given language with new operations and relations definable in {\em some way} is pervasive in Algebra and Model Theory. If we focus on operations defined by systems of equations on algebraic structures we arrive at the notion of {\em Algebraic Expansions} (\cite{CamVag09-AlgExpClasses}). Restricting to this kind of definability has the advantage of producing well-behaved expansions that can be studied with `universal-algebraic' techniques (e.g., sheaf representations). We describe these expansions in more detail. 

Let $\tau$ be an algebraic language. Given a class of $\tau$-algebras $\mathcal{K}$ and a system of equations of the form
\begin{align*}
s_{1}(x_{1},\dots,x_{n},z_{1},\dots,z_{m}) & = t_{1}(x_{1},\dots,x_{n},z_{1},\dots,z_{m}) \\
& \hspace{2mm} \vdots \\
s_{k}(x_{1},\dots,x_{n},z_{1},\dots,z_{m}) & = t_{k}(x_{1},\dots,x_{n},z_{1},\dots,z_{m})
\end{align*}
we can consider the class $\mathcal{A}$ of those algebras in $\mathcal{K}$ for which, given values for the $x$'s, there are unique values for the $z$'s such that all equalities hold. We say that $\mathcal{A}$ is an {\em Algebraically Expandable (AE) subclass} of $\mathcal{K}$ given that the members of $\mathcal{A}$ can be expanded with the operations defined by the system of equations. For example, let $\mathcal{K}$ be the class of $\{\imp,1\}$-subreducts of Boolean algebras, and consider the system of equations  
\begin{align*}
z \imp x_1 & = 1, \\
z \imp x_2 & = 1, \\
((x_1 \imp z) \imp (x_2 \imp z)) \imp (x_2 \imp z) = 1.
\end{align*}
The class $\mathcal{A}$ in this case is the class of algebras in $\mathcal{K}$ where every two elements have a meet with respect to the ordering induced by $\imp$. The expansion of $\mathcal{A}$ is (term-equivalent to) the class of generalized Boolean algebras.

In the setting of Abstract Logic expansions by new connectives are a common theme as well, in particular, expansions of a logic $L$ with connectives determined in some way by $L$.  
As we know, there is a long-standing and fruitful interplay between Logic and Algebra, so it is natural to consider what, if any, is the logical counterpart of AE-classes. As we shall see, for the case of an algebraizable logic $L$ with equivalent algebraic semantics $\mathcal{Q}$, the AE-subclasses $\mathcal{Q}$ are in correspondence with the family of a specific kind of expansions of $L$, which we call {\em algebraic expansions}. The notion of an algebraic expansion of a logic turns out to be quite natural, we think, and interestingly it falls into the general framework of expansions by implicit connectives studied by X. Caicedo in \cite{Caicedo04-ImpConnectivesAlgLogics}. An immediate consequence is that algebraic expansions are again algebraizable. The algebraic expansions of $L$ are naturally ordered by morphisms that preserve the language of $L$. It turns out that this is a lattice ordering when considered modulo equipollency, and the ensuing lattice is dually isomorphic with the lattice of AE-subclasses of $\mathcal{Q}$ under inclusion.

Besides introducing the notion of algebraic expansions of a logic we analyze two particular cases: $\ell$-groups and perfect MV-algebras. In both cases we obtain full descriptions of the AE-classes, and thus, of the algebraic expansions of their corresponding logics. We show that in both cases there is a continuum of expansions, and the lattices are isomorphic with $\mathbf{2}^\omega \oplus \mathbf{1}$ and $\mathbf{2}^\omega \oplus \mathbf{2}$, in the former and latter case respectively.

In the next section we summarize all the basic definitions and properties of the theory of AE-classes needed for this article. In Section 3 we give the formal definition of algebraic expansion of a logic, and prove the fundamental results linking them with AE-classes (Theorems \ref{TEO: piedra fundacional} and \ref{TEO: equivalencias interpretaciones}). In Section \ref{SECTION: l-groups} we characterize the AE-classes of Abelian $\ell$-groups and the algebraic expansions of their corresponding logic. Finally, in Section \ref{SECTION: hoops and MV}, we translate the results from Section 4 to their analogs for perfect MV-algebras, using cancellative hoops as an intermediate step. This completely describes the algebraic expansions of the associated logic.

\section{Preliminaries}
\label{SECTION: Preliminaries}

In this section we introduce fundamental definitions, establish notation and present several basic facts needed in the sequel. We assume familiarity with basic Universal Algebra, Model Theory and Abstract Algebraic Logic (see, e.g., \cite{BurSan81-Book-CourseUnivAlg,Hodges93-Book-ModelTheory,Font16-Book-AbstractAlgebraicLogic}, respectively).

\subsection{Notation and basic definitions}

Throughout this article algebras are considered as models of first-order languages without relations. For example, Abelian $\ell$-groups are algebras in the language $\tau_\mathcal{G} := \{+,-,0,\vee,\wedge\}$.  As is customary we use bold letters ($\mathbf{A}, \mathbf{B}, \mathbf{C}, \ldots$) for algebraic structures and italic letters ($A, B, C, \ldots$) for the underlying sets. For algebras $\mathbf{A}$ and $\mathbf{B}$ we write $\mathbf{A} \subseteq \mathbf{B}$ whenever $\mathbf{A}$ is a subalgebra of $\mathbf{B}$. 

Given a structure $\mathbf{A}$ in a language $\tau$ and a term $t(x_1,\ldots,x_n)$ in the same language, we write $t^\mathbf{A}(\bar{a})$ for the value of the term upon assigning elements $a_1,\ldots,a_n$ from $A$ to the variables $x_1,\ldots,x_n$. We may omit the superscript $\mathbf{A}$ if there is no risk of confusion.

Given a (first-order) formula $\varphi$, we say that $\varphi$ is
\begin{itemize} \setlength{\itemsep}{0pt}\setlength{\parskip}{0pt}
\item an {\em identity} if it has the form $\forall \bar{x} (p(\bar{x}) = q(\bar{x}))$, where $p$ and $q$ are terms,
\item a {\em quasi-identity} if it has the form $\forall \bar{x} (\alpha(\bar{x}) \bimp \beta(\bar{x}))$,\footnote{We write the first-order connectives $\band$, $\bor$, $\bimp$, $\biff$ in bold font to distinguish them from algebraic operations and connectives in sentential logics.}  where $\alpha$ is a finite conjunction of term-equalities and $\beta$ is a term-equality,
\item {\em universal} if it has the form $\forall \bar{x} \psi$, where $\psi$ is quantifier-free,
\item {\em existential} if it has the form $\exists \bar{x} \psi$, where $\psi$ is quantifier-free.
\end{itemize}
A {\em sentence} is a formula with no free variables. If $\Sigma$ is a set of sentences, $\Mod(\Sigma)$ denotes the class of all models that satisfy the sentences in $\Sigma$. 

Whenever we consider a class $\mathcal{K}$ of algebras, we assume that all algebras in $\mathcal{K}$ have the same language. Given a class $\mathcal{K}$ of algebras, we define the usual class operators:
\begin{itemize} \setlength{\itemsep}{0pt}\setlength{\parskip}{0pt}
\item $\mathsf{I}(\mathcal{K})$ denotes the class of isomorphic images of members of $\mathcal{K}$,
\item $\mathsf{H}(\mathcal{K})$ denotes the class of homomorphic images of members of $\mathcal{K}$,
\item $\mathsf{S}(\mathcal{K})$ denotes the class of subalgebras of members of $\mathcal{K}$,
\item $\mathsf{P}(\mathcal{K})$ denotes the class of direct products with factors in $\mathcal{K}$,
\item $\mathsf{P_U}(\mathcal{K})$ denotes the class of ultraproducts with factors in $\mathcal{K}$.
\end{itemize}
If $\mathsf{O}$ is one of the above operators and $\mathcal{K} = \{\mathbf{A}_1,\ldots, \mathbf{A}_n\}$, we write $\mathsf{O}(\mathbf{A}_1,\ldots,\mathbf{A}_n)$ instead of $\mathsf{O}(\mathcal{K})$. 

Remember that a class $\mathcal{K}$ is a {\em variety} (or {\em equational class}) if it can be axiomatized using a set of identities; equivalently, $\mathcal{K}$ is a variety if it is closed under $\mathsf{H}$, $\mathsf{S}$ and $\mathsf{P}$. The smallest variety containing a given class $\mathcal{K}$ is $\mathsf{HSP}(\mathcal{K})$ and is denoted by $\mathsf{V}(\mathcal{K})$. A {\em quasivariety} is a class of algebras axiomatized by a set of quasi-identities. A class $\mathcal{K}$ is a quasivariety if and only if $\mathcal{K}$ is closed under $\mathsf{I}$, $\mathsf{S}$, $\mathsf{P}$ and $\mathsf{P_U}$; the smallest quasivariety containing a given class $\mathcal{K}$ is $\mathsf{ISPP_U}(\mathcal{K})$, also denoted by $\mathsf{Q}(\mathcal{K})$. Finally, recall that $\mathcal{K}$ is {\em universal} if it is axiomatized by universal sentences, which is equivalent to $\mathcal{K}$ being closed under $\mathsf{I}$, $\mathsf{S}$ and $\mathsf{P_U}$. Moreover, the smallest universal class containing a class $\mathcal{K}$ is given by $\mathsf{ISP_U}(\mathcal{K})$.

Given a class $\mathcal{K}$ and two sentences $\varphi$, $\psi$, we say that $\varphi$ and $\psi$ are {\em equivalent} in $\mathcal{K}$, and write $\varphi \eq \psi$ in $\mathcal{K}$, if for every $\mathbf{A} \in \mathcal{K}$ we have that $\mathbf{A} \vDash \varphi$ if and only if $\mathbf{A} \vDash \psi$.

\subsection{Algebraically expandable classes}
\label{SUBSECTION:AE-classes}

In order to define algebraically expandable classes \cite{CamVag09-AlgExpClasses}, one of the fundamental notions in this article, we need to define the special type of sentences that define them. An {\em equational function definition sentence} (EFD-sentence for short) in the language $\tau$ is a sentence of the form\begin{equation} \forall x_1 \ldots x_n \exists! z_1 \ldots z_m \, \bigand_{i=1}^k s_i(\bar{x},\bar{z}) = t_i(\bar{x},\bar{z}) \label{def phi} \end{equation} where $s_i,t_i$ are $\tau$-terms, $n \geq 0$, and $m \geq 1$. Suppose $\varphi$ is the EFD-sentence in \eqref{def phi}. Observe that $\varphi$ is valid in a structure $\mathbf{A}$ if and only if the system of equations $\displaystyle \bigand_{i=1}^k s_i(\bar{x},\bar{z}) = t_i(\bar{x},\bar{z})$ defines a (total) function $F\colon A^n \to A^m$. We denote the coordinate functions of $F$ by $[\varphi]_1^\mathbf{A}, \ldots, [\varphi]_m^\mathbf{A}$.

Let $\varphi$ be as in \eqref{def phi}. We define:
\begin{itemize} \setlength{\itemsep}{0pt}\setlength{\parskip}{1pt}
\item $E(\varphi) := \displaystyle \forall \bar{x} \exists \bar{z} \, \bigand_{i=1}^k s_i(\bar{x},\bar{z}) = t_i(\bar{x},\bar{z})$,
\item $U(\varphi) := \displaystyle \forall \bar{x} \bar{y} \bar{z} \bigand_{i=1}^k s_i(\bar{x},\bar{y}) = t_i(\bar{x},\bar{y}) \band \bigand_{i=1}^k s_i(\bar{x},\bar{z}) = t_i(\bar{x},\bar{z}) \bimp \bar{y} = \bar{z}$.
\end{itemize}
The following basic facts are used without explicit reference throughout the article.
\begin{itemize}
\item $\varphi$ is equivalent to $E(\varphi) \band U(\varphi)$.
\item $U(\varphi)$ is (equivalent to) a conjunction of quasi-identities.
\item $E(\varphi)$ is preserved by homomorphic images, that is, for any surjective homomorphism $f: \mathbf{A} \to \mathbf{B}$, if $\mathbf{A} \vDash E(\varphi)$, then $\mathbf{B} \vDash E(\varphi)$.
\end{itemize}

A class of algebras $\mathcal{K}$ is an {\em algebraically expandable class} (AE-class for short) if there is a set of EFD-sentences $\Sigma$ such that $\mathcal{K} = \Mod(\Sigma)$. Let $\mathcal{K}$ and $\mathcal{C}$ be classes of algebras, $\mathcal{K} \subseteq \mathcal{C}$. We say that $\mathcal{K}$ is an {\em AE-subclass} of $\mathcal{C}$ if $\mathcal{K}$ is axiomatizable by EFD-sentences relative to $\mathcal{C}$, that is, $\mathcal{K} = \mathcal{C} \cap \Mod(\Sigma)$ for some set $\Sigma$ of EFD-sentences. The reader should be aware that $\mathcal{K}$ may be an AE-subclass of $\mathcal{C}$,
but fail to be an AE-class itself.

Let $\mathcal{Q}$ be a quasivariety in the language $\tau$ and let $\Sigma$ be a set of EFD-sentences. There is an obvious expansion of the AE-subclass $\mathcal{K} := \mathcal{Q} \cap \Mod(\Sigma)$ of $\mathcal{Q}$ obtained by skolemizing the existential quantifiers in $\Sigma$; details follow. For each $\varphi \in \Sigma$ of the form $\forall x_1 \ldots x_n \exists! z_1 \ldots z_m \, \bigand_{i=1}^k s_i(\bar{x},\bar{z}) = t_i(\bar{x},\bar{z})$ consider new $n$-ary function symbols $f_1^\varphi, \ldots, f_m^\varphi$ and the set of identities $$E_\varphi := \{\forall \bar{x} \, s_i(\bar{x}, f_1^\varphi(\bar{x}), \ldots, f_m^\varphi(\bar{x})) = t_i(\bar{x}, f_1^\varphi(\bar{x}), \ldots, f_m^\varphi(\bar{x})): 1 \leq i \leq k\}.$$ Let $\tau_\Sigma$ be the expansion of $\tau$ obtained by adding the $f_j^\varphi$'s for each $\varphi \in \Sigma$, and put
\begin{align*}
E_\Sigma & := \bigcup_{\varphi \in \Sigma} E_\varphi, \\
U_\Sigma & := \{U(\varphi): \varphi \in \Sigma\}.
\end{align*}
Define $$\mathcal{Q}^\Sigma := \Mod(\Gamma \cup E_\Sigma \cup U_\Sigma),$$ where $\Gamma$ is a set of quasi-identities axiomatizing $\mathcal{Q}$. We call $\mathcal{Q}^\Sigma$ an {\em algebraic expansion of} $\mathcal{Q}$.  Note that $\mathcal{Q}^\Sigma$ is a quasivariety over the language $\tau_\Sigma$ whose members are precisely the expansions of the members of $\mathcal{K}$. An interesting property of this process is that, up to term-equivalence, the quasivariety $\mathcal{Q}^\Sigma$ is determined by the AE-subclass $\mathcal{K}$ and not by the axiomatization $\Sigma$; for details see Theorem \ref{TEO: equivalencias interpretaciones}.

%

We conclude this section with a preservation result for EFD-sentences needed in the sequel. Recall that a structure $\mathbf{A}$ is {\em finitely subdirectly irreducible} if its diagonal congruence is finitely meet-irreducible in the congruence lattice of $\mathbf{A}$. We write $\mathcal{K}_\mathrm{fsi}$ for the class of finitely subdirectly irreducible members of $\mathcal{K}$. A variety is {\em arithmetical} provided that it is both congruence distributive and congruence permutable.

\begin{lemma} \label{LEMA: aritm + fsi univ}
Let $\mathcal{V}$ be an arithmetical variety such that $\mathcal{V}_\mathrm{fsi} \cup \{\text{trivial algebras}\}$ is a universal class, and let $\mathbf{A} \in \mathcal{V}$. If $\varphi$ is an EFD-sentence such that $\mathsf{H}(\mathbf{A})_\mathrm{fsi} \vDash \varphi$, then $\mathbf{A} \vDash \varphi$.
\end{lemma}

\begin{proof}
By \cite{GraVag96-SheafRepVarLatExp} $\mathbf{A}$ has a global representation with factors in $\mathsf{H}(\mathbf{A})_\mathrm{fsi}$, and by \cite{Volger79-PreservThmLimitsStrucGlobalSections} global representations preserve EFD-sentences.
\end{proof}

\section{The algebraic expansions of a logic}

Throughout this work, by a (sentential) {\em logic} over a language $\tau$ we mean a finitary structural consequence operator on the set of $\tau$-formulas. We refer the reader to \cite{Font16-Book-AbstractAlgebraicLogic} for definitions and results about abstract algebraic logic not explicitly mentioned in this article.

Let $L$ be an algebraizable logic and let $\Delta(x,y)$ be a set of equivalence formulas for $L$. Given a finite set of $L$-formulas $\Phi(\bar{x},\bar{z})$ on variables $x_1,\ldots,x_n, z_1,\ldots,z_m$, $n,m\in \omega$, let $f^\Phi_1, \ldots, f^\Phi_m$ be new $n$-ary symbols and let $L^\Phi$ be the least propositional logic containing $L$ such that:
\begin{align*}
& \vdash_{L^\Phi} \Phi(\bar{x},f^\Phi_1(\bar{x}), \ldots, f^\Phi_m(\bar{x})), \tag{$\mathrm{E}_\Phi$} \\
& \Phi(\bar{x},\bar{y}), \Phi(\bar{x},\bar{z}) \vdash_{L^\Phi} \Delta(\bar{y},\bar{z}). \tag{$\mathrm{U}_\Phi$}
\end{align*}
($\Delta(\bar{y},\bar{z})$ is shorthand for $\bigcup_{j=1}^m \Delta(y_j,z_j)$.) We say that $L^\Phi$ is the {\em algebraic expansion} of $L$ by $\Phi$. Recall that if $\Delta'(x,y)$ is another set of equivalence formulas for $L$, then $\Delta(x,y) \dashv\vdash_L \Delta'(x,y)$. Thus the expansion $L^\Phi$ does not depend on the choice of equivalence formulas.

Given a set $\Sigma$ of finite sets of $L$-formulas,  define $L^\Sigma$ as the least propositional logic containing $L^\Phi$ for every $\Phi \in \Sigma$. (Of course, we assume that the new symbols for each $L^\Phi$ are different.) The logic $L^\Sigma$ is called the {\em algebraic expansion} of $L$ by $\Sigma$.

Observe that, in the definition of $L^\Phi$, for the case $m = 0$ condition $\mathrm{U}_\Phi$ holds vacuously, so $L^\Phi$ is just the axiomatic extension of $L$ by $\Phi$. Hence, axiomatic extensions of $L$ are algebraic expansions of $L$. 

As mentioned in the introduction, in \cite{Caicedo04-ImpConnectivesAlgLogics} Caicedo studies expansions of finitely algebraizable logics where the behaviour of the new connectives is determined by the added axioms and rules. We call such an expansion of a logic $L$ an {\em expansion of $L$ by implicit connectives}. 

It is easy to see that the expansion $L^\Sigma$ defined above is in fact an expansion of $L$ by implicit connectives (where $(\mathrm{E}_\Phi)$ and $(\mathrm{U}_\Phi)$ correspond to new axioms and rules, respectively). As an immediate consequence of this fact we have that $L^\Sigma$ is algebraizable with the same equivalence formulas and defining equations as $L$ \cite[Theorem 1]{Caicedo04-ImpConnectivesAlgLogics}. Furthermore, the equivalent algebraic semantics of $L^\Sigma$ is the expected one \cite[Corollary 2]{Caicedo04-ImpConnectivesAlgLogics}, which in this case turns out to be an algebraic expansion of the equivalent algebraic semantics of $L$. Details follow.

Let $\mathcal{Q}$ be the equivalent algebraic semantics of $L$ via the set of equivalence formulas $\Delta(x,y)$ and the set of defining equations $\varepsilon(x)$. Given a finite set $\Phi(\bar{x},\bar{z})$ of $L$-formulas, let $e(\Phi)$ be the EFD-sentence $\forall \bar{x} \exists! \bar{z} \, \bigand \varepsilon(\Phi(\bar{x},\bar{z}))$\footnote{To avoid confusion with the connectives of the logic we use $\bigwedge$ for the logical conjunction.}. For $\Sigma$ a set of finite sets of $L$-formulas define $e(\Sigma) := \{e(\Phi): \Phi \in \Sigma\}$. Now, Corollary 2 in \cite{Caicedo04-ImpConnectivesAlgLogics} says that the algebraic expansion $\mathcal{Q}^{e(\Sigma)}$ is the equivalent algebraic semantics of $L^\Sigma$. Thus, for each algebraic expansion of $L$ we have a corresponding algebraic expansion of $\mathcal{Q}$. Of course, we can also go in the other direction. Given an EFD-sentence $\varphi := \forall \bar{x} \exists! \bar{z} \, \alpha(\bar{x},\bar{z})$, put $d(\varphi) := \Delta(\alpha(\bar{x},\bar{z}))$. Here and in the sequel $\Delta(\alpha(\bar{x},\bar{z}))$ abbreviates $\bigcup_{i=1}^k \Delta(s_i(\bar{x},\bar{z}),t_i(\bar{x},\bar{z}))$ if $\alpha(\bar{x},\bar{z})$ is the conjunction of equations $\bigand_{i=1}^k s_i(\bar{x},\bar{z}) = t_i(\bar{x},\bar{z})$. For a set $\Sigma$ of EFD-sentences, we write $d(\Sigma)$ for the set $\{d(\varphi): \varphi \in \Sigma\}$. Again, it is straightforward to check that $\mathcal{Q}^\Sigma$ is the equivalent algebraic semantics of $L^{d(\Sigma)}$. Furthermore,
\begin{itemize}
\item $L^\Sigma = L^{d(e(\Sigma))}$,
\item $\mathcal{Q}^\Sigma = \mathcal{Q}^{e(d(\Sigma))}$
\end{itemize}
for suitable $\Sigma$'s. This establishes a direct correspondence between algebraic expansions of a logic and those of its equivalent algebraic semantics. Theorem \ref{TEO: equivalencias interpretaciones} below explores this connection in greater detail. In the sequel, to avoid cumbersome notation, given a logic $L$ and a set $\Sigma$ of EFD-sentences we write $L^\Sigma$ instead of $L^{d(\Sigma)}$.

For future reference, the facts above are summarized in the following:

\begin{theorem} \label{TEO: piedra fundacional}
Let $L$ be a finitely algebraizable logic with equivalent algebraic semantics $\mathcal{Q}$. Let $\Sigma$ be a set of EFD-sentences in the language of $\mathcal{Q}$. Then the algebraic expansion $L^\Sigma$ is algebraizable with the same equivalence formulas and defining equations as $L$, and its equivalent algebraic semantics is the quasivariety $\mathcal{Q}^\Sigma$. Moreover, there is a one-to-one correspondence between the algebraic expansions of $L$ and the algebraic expansions of $\mathcal{Q}$.
\end{theorem}

We conclude this section with an example of a logic that has an expansion by implicit connectives that is not algebraic. Let $L_{\mathrm{int}}$ be the Intuitionistic Logic and let $L_\mathrm{int}^S$ be the extension of $L_\mathrm{int}$ by the implicit connective $S$ defined in \cite[Example 5.2]{CaiCig01-AlgApprIntuitConnectives}. The equivalent algebraic semantics of $L_\mathrm{int}^S$ is the variety $\mathcal{H}^S$ of Heyting algebras with successor. It is not hard to show that the class of Heyting-reducts of algebras in $\mathcal{H}^S$ is not an AE-subclass of $\mathcal{H}$. Thus, by Theorem \ref{TEO: piedra fundacional}, $L_\mathrm{int}^S$ cannot be an algebraic expansion of $L_\mathrm{int}$.

\subsection{The lattice of algebraic expansions}

Let $L$ be a finitely algebraizable logic with equivalent algebraic semantics $\mathcal{Q}$. 
The AE-subclasses of the quasivariety $\mathcal{Q}$ are naturally (lattice-)ordered by inclusion. In the current section we show how this ordering translates to the algebraic expansions of $\mathcal{Q}$, and thus to the algebraic expansions of $L$. For this we need to look into interpretations between logics and between classes of algebras.

Fix a countably infinite set of variables $X := \{x_1,x_2,\ldots\}$; given a language $\tau$ we write $Tm(\tau)$ for the set of $\tau$-terms over the variables in $X$. Let $\tau_1$ and $\tau_2$ be two expansions of a language $\tau$. A {\em $\tau$-translation} from $\tau_1$ into $\tau_2$ is a function $T\colon \tau_1 \to Tm(\tau_2)$ such that $T$ maps each symbol of arity $n$ to a term in the variables $x_1,\ldots,x_n$, and $T(f) = f(x_1,\ldots,x_n)$ for every $n$-ary symbol $f \in \tau$.

Let $\mathcal{K}_1$ and $\mathcal{K}_2$ be two classes of algebras over $\tau_1$ and $\tau_2$, respectively. A {\em $\tau$-interpretation} of $\mathcal{K}_1$ in $\mathcal{K}_2$ is a $\tau$-translation $T\colon \tau_1 \to Tm(\tau_2)$ such that for every member $\mathbf{A} := (A,\{g^\mathbf{A}: g \in \tau_2\})$ in $\mathcal{K}_2$, the algebra $\mathbf{A}^T := (A,\{T(f)^\mathbf{A}: f \in \tau_1\})$ belongs to $\mathcal{K}_1$. 
If $T$ and $S$ are $\tau$-interpretations of $\mathcal{K}_1$ in $\mathcal{K}_2$ and $\mathcal{K}_2$ in $\mathcal{K}_1$, respectively, such that the maps $\mathbf{A} \mapsto \mathbf{A}^T$ and $\mathbf{A} \mapsto \mathbf{A}^S$ are mutually inverse, we say that $\mathcal{K}_1$ and $\mathcal{K}_2$ are {\em $\tau$-term-equivalent}.

We turn now to maps between logics. A $\tau$-translation $T$ from $\tau_1$ into $\tau_2$ extends in a natural way to a mapping from $Tm(\tau_1)$ to $Tm(\tau_2)$:
\begin{itemize}
\item $T(x) = x$ for every variable $x \in X$;
\item $T(f(\varphi_1,\ldots,\varphi_n)) = T(f)(T(\varphi_1),\ldots,T(\varphi_n))$ for $f$ in $\tau_1$ of arity $n$ and $\varphi_1, \ldots \varphi_n$ in $Tm(\tau_1)$.
\end{itemize}
Given a set $\Gamma$ of $\tau_1$-terms we write $T(\Gamma)$ for $\{T(\varphi): \varphi \in \Gamma\}$.

Let $\tau_1$ and $\tau_2$ be expansions of a language $\tau$, and suppose $L_1$ and $L_2$ are logics in $\tau_1$ and $\tau_2$, respectively. A {\em $\tau$-morphism} from $L_1$ to $L_2$ is a $\tau$-translation from $\tau_1$ into $\tau_2$ such that $$\Gamma \vdash_{L_1} \varphi \text{ implies } T(\Gamma) \vdash_{L_2} T(\varphi)$$ for $\Gamma \cup \{\varphi\} \subseteq Tm(\tau_1)$. We say that $L_1$ and $L_2$ are {\em $\tau$-equipollent} (cf. \cite{CalGon05-EquipollentLogSys}) provided there are $\tau$-morphisms $T$ from $L_1$ to $L_2$ and $S$ from $L_2$ to $L_1$ such that:
\begin{itemize}
\item $\varphi \dashv\vdash_{L_1} S(T(\varphi))$ for every $\varphi \in Tm(\tau_1)$;
\item $\varphi \dashv\vdash_{L_2} T(S(\varphi))$ for every $\varphi \in Tm(\tau_2)$.
\end{itemize}

The following result shows the connection between the above defined relations.

\begin{theorem}\label{TEO: equivalencias interpretaciones}
Let $L$ be a finitely algebraizable logic in the language $\tau$ with equivalent algebraic semantics $\mathcal{Q}$. Let $\Sigma$ and $\Sigma'$ be two sets of EFD-sentences in $\tau$.
\begin{enumerate}
\item The following are equivalent:
\begin{enumerate}[$(i)$]
\item there is a $\tau$-morphism from $L^{\Sigma'}$ to $L^\Sigma$,
\item there is a $\tau$-interpretation of $\mathcal{Q}^{\Sigma'}$ in $\mathcal{Q}^\Sigma$,
\item $\mathcal{Q} \cap \Mod(\Sigma) \subseteq \mathcal{Q} \cap \Mod(\Sigma')$.
\end{enumerate}
\item The following are equivalent:
\begin{enumerate}[$(i)$]
\item $L^{\Sigma'}$ and $L^\Sigma$ are $\tau$-equipollent,
\item $\mathcal{Q}^{\Sigma'}$ and $\mathcal{Q}^\Sigma$ are $\tau$-term-equivalent,
\item $\mathcal{Q} \cap \Mod(\Sigma) = \mathcal{Q} \cap \Mod(\Sigma')$.
\end{enumerate}
\end{enumerate}
\end{theorem}

\begin{proof}
Fix a finite set of equivalence formulas $\Delta(x,y)$ and a finite set of defining equations $\varepsilon(x)$ that witness the algebraization relation between $L$ and $\mathcal{Q}$. From Theorem \ref{TEO: piedra fundacional} we know that the same sets algebraize $L^\Sigma$ and $L^{\Sigma'}$ with $\mathcal{Q}^{\Sigma}$ and $\mathcal{Q}^{\Sigma'}$ as their corresponding equivalent algebraic semantics.

Let us start by proving 1.

$(i)\mathord{\Rightarrow}(ii)$. Assume first there is a $\tau$-morphism $T$ from $L^{\Sigma'}$ to $L^\Sigma$. We prove that $T$ is also a $\tau$-interpretation of $\mathcal{Q}^{\Sigma'}$ in $\mathcal{Q}^\Sigma$. Let $\mathbf{A} \in \mathcal{Q}^\Sigma$; we aim to prove that $\mathbf{A}^T \in \mathcal{Q}^{\Sigma'}$. Since the $\tau$-reducts of $\mathbf{A}^T$ and $\mathbf{A}$ coincide, the algebra $\mathbf{A}^T$ satisfies the quasi-identities valid in $\mathcal{Q}$. Let $\varphi := \forall \bar{x} \exists! \bar{z} \; \alpha(\bar{x},\bar{z})$ be an EFD-sentence in $\Sigma'$. We show that $\mathbf{A}^T$ satisfies the identities $E_\varphi$ and the quasi-identity $U(\varphi)$. We start by showing that $\mathbf{A}^T$ satisfies $E_\varphi$. By definition, we have $\vdash_{L^{\Sigma'}} \Delta(\alpha(\bar{x},f_1^\varphi(\bar{x}), \ldots, f_m^\varphi(\bar{x}))$ and, since $T$ is a $\tau$-morphism, $\vdash_{L^\Sigma} T(\Delta(\alpha(\bar{x},f_1^\varphi(\bar{x}), \ldots, f_m^\varphi(\bar{x})))$. Both $\Delta$ and $\alpha$ are in the language $\tau$, so $\vdash_{L^\Sigma} \Delta(\alpha(\bar{x},T(f_1^\varphi)(\bar{x}), \ldots, T(f_m^\varphi)(\bar{x})))$. Then $\mathcal{Q}^\Sigma \vDash \forall \bar{x} \,  \alpha(\bar{x},T(f_m^\varphi)(\bar{x}), \ldots, T(f_m^\varphi)(\bar{x}))$, and thus $\mathbf{A}^T \vDash \forall \bar{x} \, \alpha(\bar{x},f_1^\varphi(\bar{x}), \ldots, f_m^\varphi(\bar{x}))$. This shows that $\mathbf{A}^T \vDash E_\varphi$. We show next that $\mathbf{A}^T$ satisfies $U(\varphi)$. Again, by definition, we know that $\Delta(\alpha(\bar{x},\bar{y})) \cup \Delta(\alpha(\bar{x},\bar{z})) \vdash_{L^{\Sigma'}} \Delta(\bar{y},\bar{z})$. Applying the $\tau$-morphism $T$ to this deduction immediately produces $\Delta(\alpha(\bar{x},\bar{y})) \cup \Delta(\alpha(\bar{x},\bar{z})) \vdash_{L^\Sigma} \Delta(\bar{y},\bar{z})$ since all formulas involved are in $Tm(\tau)$. By the algebraization relation $\mathcal{Q}^\Sigma \vDash U(\varphi)$ and, in particular, $\mathbf{A} \vDash U(\varphi)$. Again, noting that $\mathbf{A}$ and $\mathbf{A}^T$ have the same $\tau$-reduct, we get that $\mathbf{A}^T \vDash U(\varphi)$.

$(ii)\mathord{\Rightarrow}(i)$. Let $T$ be a $\tau$-interpretation from $\mathcal{Q}^{\Sigma'}$ in $\mathcal{Q}^\Sigma$. We claim that $T$ is a $\tau$-morphism from $L^{\Sigma'}$ to $L^\Sigma$. By the algebraizability relation, this amounts to showing that
\begin{equation} \varepsilon(\Gamma) \vDash_{\mathcal{Q}^{\Sigma'}} \varepsilon(\varphi) \text{ implies } \varepsilon(T(\Gamma)) \vDash_{\mathcal{Q}^{\Sigma}} \varepsilon(T(\varphi)) \tag{$*$}
\end{equation}
for every $\Gamma \cup \{\varphi\} \subseteq Tm(\tau_{\Sigma'})$. Since the defining equations $\varepsilon(x)$ only use symbols from $\tau$, we have that $(*)$ is equivalent to
\begin{equation}
\varepsilon(\Gamma) \vDash_{\mathcal{Q}^{\Sigma'}} \varepsilon(\varphi) \text{ implies } T(\varepsilon(\Gamma)) \vDash_{\mathcal{Q}^{\Sigma}}  T(\varepsilon(\varphi)). \tag{$**$}
\end{equation}
Now, observe that for all algebras $\mathbf{A} \in \mathcal{Q}^\Sigma$, all $\tau_{\Sigma'}$-terms $t(\bar{x})$ and all tuples $\bar{a}$ from $\mathbf{A}$ we have that $T(t)^{\mathbf{A}}(\bar{a}) = t^{\mathbf{A}^T}(\bar{a})$. From this fact it is straightforward to prove $(**)$.

$(ii)\mathord{\Rightarrow}(iii)$. Suppose $T$ is a $\tau$-interpretation from $\mathcal{Q}^{\Sigma'}$ in $\mathcal{Q}^\Sigma$, and let $\mathbf{A} \in \mathcal{Q} \cap \Mod(\Sigma)$. Let $\mathbf{A}^\Sigma$ the expansion of $\mathbf{A}$ in $\mathcal{Q}^\Sigma$. Then $(\mathbf{A}^\Sigma)^T \in \mathcal{Q}^{\Sigma'}$. Since $(\mathbf{A}^\Sigma)^T$ satisfies $\Sigma'$ and $\Sigma'$ is a set of $\tau$-sentences, $\mathbf{A}$ satisfies $\Sigma'$ as well. This shows that $\mathbf{A} \in \mathcal{Q} \cap \Mod(\Sigma')$.

$(iii)\mathord{\Rightarrow}(ii)$. This follows from the proof of \cite[Theorem 5]{CamVag11-AlgebraicFunctions}.

We turn now to the equivalences in 2.

%

$(i)\mathord{\Rightarrow}(iii)$. This is immediate from 1.

$(iii)\mathord{\Rightarrow}(ii)$. Assume $\mathcal{Q} \cap \Mod(\Sigma) = \mathcal{Q} \cap \Mod(\Sigma')$. From 1. there are $\tau$-interpretations $T$ and $S$ from $\mathcal{Q}^{\Sigma'}$ in $\mathcal{Q}^\Sigma$ and from $\mathcal{Q}^{\Sigma}$ in $\mathcal{Q}^{\Sigma'}$, respectively. Fix $\mathbf{A}$ in $\mathcal{Q}^\Sigma$, and let $f$ be a symbol in $\tau_\Sigma \setminus \tau$.  By the definition of $\mathcal{Q}^\Sigma$, there is an EFD-sentence $\varphi := \forall \bar{x} \exists! \bar{z} \, \alpha(\bar{x},\bar{z})$ in $\Sigma$ such that $f = f^\varphi_i$. Put $\mathbf{B} := (\mathbf{A}^T)^S$ and note that it suffices to prove that $f^\mathbf{A} = f^\mathbf{B}$. Fix a sequence $\bar{a}$ of elements from $\mathbf{A}$. Since $\mathbf{B} \vDash E_\varphi$, we have that $\mathbf{B} \vDash \alpha(\bar{a},(f^\varphi_1)^\mathbf{B}(\bar{a}),\ldots,(f^\varphi_m)^\mathbf{B}(\bar{a}))$. As $\mathbf{A}$ and $\mathbf{B}$ have the same $\tau$-reduct and $\alpha$ is a $\tau$-formula, it follows that $\mathbf{A} \vDash \alpha(\bar{a},(f^\varphi_1)^\mathbf{B}(\bar{a}),\ldots,(f^\varphi_m)^\mathbf{B}(\bar{a}))$. We also know that $\mathbf{A} \vDash E_\varphi$, and thus $\mathbf{A} \vDash \alpha(\bar{a},(f^\varphi_1)^\mathbf{A}(\bar{a}),\ldots,(f^\varphi_m)^\mathbf{A}(\bar{a}))$. Since $\mathbf{A} \vDash U(\varphi)$, we conclude that $(f_i^\varphi)^\mathbf{A}(\bar{a}) = (f_i^\varphi)^\mathbf{B}(\bar{a})$. We proved that $(\mathbf{A}^T)^S = \mathbf{A}$. Analogously $(\mathbf{A}^S)^T = \mathbf{A}$ for every $\mathbf{A} \in \mathcal{Q}^{\Sigma'}$.

$(ii)\mathord{\Rightarrow}(i)$. Suppose $\mathcal{Q}^{\Sigma'}$ and $\mathcal{Q}^\Sigma$ are $\tau$-term-equivalent and let $T$ and $S$ be $\tau$-interpretations such that $(\mathbf{A}^S)^T = \mathbf{A}$ for every $\mathbf{A} \in \mathcal{Q}^{\Sigma'}$ and $(\mathbf{A}^T)^S = \mathbf{A}$ for every $\mathbf{A} \in \mathcal{Q}^\Sigma$. We claim that $T$ and $S$ make $L^{\Sigma'}$ and $L^\Sigma$ $\tau$-equipollent. By 1., the maps $T$ and $S$ are $\tau$-morphisms from $L^{\Sigma'}$ in $L^\Sigma$ and from $L^\Sigma$ to $L^{\Sigma'}$, respectively. It remains to show that $\varphi \dashv\vdash_{L^{\Sigma'}} S(T(\varphi))$ for every $\varphi \in Tm(\tau_{\Sigma'})$ and $\varphi \dashv\vdash_{L^\Sigma} T(S(\varphi))$ for every $\varphi \in Tm(\tau_\Sigma)$. We prove the first equivalence, the second one being analogous. By the algebraizability relation, it is enough to prove that $\varepsilon(\varphi) \Dashv\vDash_{\mathcal{Q}^{\Sigma'}} \varepsilon(S(T(\varphi)))$, or equivalently, $\varepsilon(\varphi) \Dashv\vDash_{\mathcal{Q}^{\Sigma'}} S(T(\varepsilon(\varphi)))$. In fact, we claim that $\gamma \Dashv\vDash_{\mathcal{Q}^{\Sigma'}} S(T(\gamma))$ for every $\tau_{\Sigma'}$-equation $\gamma$. Indeed, for any $\mathbf{A}$ in $\mathcal{Q}^{\Sigma'}$ and any tuple $\bar{a}$ from $\mathbf{A}$ we have that $\mathbf{A} \vDash \gamma(\bar{a})$ iff $(\mathbf{A}^S)^T \vDash \gamma(\bar{a})$ iff $\mathbf{A}^S \vDash T(\gamma)(\bar{a})$ iff $\mathbf{A} \vDash S(T(\gamma))(\bar{a})$.
\end{proof}


Let $L$ be a logic algebraized by a quasivariety $\mathcal{Q}$. As is the case for any quasivariety, the AE subclasses of $\mathcal{Q}$ form a lattice $\Lambda$ under inclusion. In the light of Theorem \ref{TEO: equivalencias interpretaciones}, the algebraic expansions of $L$ modulo equipollency and ordered by morphisms form a lattice as well, dually isomorphic with $\Lambda$. Thus, classifying the algebraically expandable classes of $\mathcal{Q}$ yields a classification of all algebraic expansions of $L$ up to equipollency.

\subsection{Some examples}

When the AE-subclasses of a quasivariety are known, Theorem \ref{TEO: equivalencias interpretaciones} immediately gives the description of the algebraic expansions of the corresponding logic. We present here three examples.

\subsubsection{The primal case.}

An algebra $\mathbf{A}$ is called {\em primal} if it is finite and every function $f\colon A^n \to A$ for $n \geq 1$ is a term-operation of $\mathbf{A}$. It is proved in \cite{CamVag09-AlgExpClasses} that the only AE-subclasses of $\mathsf{V}(\mathbf{A})$ for a primal $\mathbf{A}$ are $\mathsf{V}(\mathbf{A})$ and the class of trivial algebras. Thus, the only (modulo equipollency) algebraic expansions of a logic $L$ algebraized by such a variety are $L$ itself and the inconsistent logic. This applies, e.g., to Classical Propositional Logic and $m$-valued Post's logic.

\subsubsection{Gödel logic.}

Recall that Gödel Logic $L_G$ is the extension of Intuitionistic Logic by the prelinearity axiom $(x \imp y) \vee (y \imp x)$. It is known that the equivalent algebraic semantics of $L_G$ is the variety $\mathcal{H}_G$ of Gödel algebras, also known as prelinear Heyting algebras. The only AE-subclasses of $\mathcal{H}_G$ are its subvarieties (\cite{Campercholi10-Heyting-preprint}). Thus, the algebraic expansions of $L_G$ agree with its axiomatic extensions.

\subsubsection{The implicative fragment of classical logic.}

Let $L_\to$ be the implicative fragment of classical propositional logic. The equivalent algebraic semantics of $L_\to$ is the variety $\mathcal{I}$ of implication algebras. Recall that disjunction is expressible in terms of $\imp$, thus for $n \geq 2$ and $1 \leq i \leq n$ $$s_i^n(x_1,\ldots,x_n) := \bigvee_{j=1,j\ne i}^n x_j$$ is an $\{\imp\}$-term.
For each $n \geq 2$ let $$\Phi_n := \{z \imp s_i^n(\bar{x}): i \in \{1,\ldots,n\}\} \cup \{\bigvee_{i=1}^n (s_i^n(\bar{x}) \imp z)\}.$$ By definition, $L_\to^{\Phi_n}$ is the least expansion of $L_\to$ that satisfies $(\mathrm{E}_{\Phi_n})$ and $(\mathrm{U}_{\Phi_n})$. However, condition $(\mathrm{U}_{\Phi_n})$ is already true for $L_\to$. Thus $L_\to^{\Phi_n}$ is the expansion of $L_\to$ by the following axioms
\begin{itemize}
\item[] $\mu_n(\bar{x}) \imp s_i^n(\bar{x})$ for $i \in \{1,\ldots,n\}$,
\item[] $\bigvee_{i=1}^n (s_i^n(\bar{x}) \imp \mu_n(\bar{x}))$,
\end{itemize}
where $\mu_n$ is a new $n$-ary symbol.

By the characterization of the AE-subclasses of $\mathcal{I}$ given in \cite{Campercholi10-Implication} it follows from Theorem \ref{TEO: equivalencias interpretaciones} that, up to equipollency, the consistent algebraic expansions of $L_\to$ are
$$L_\to < \ldots < L_\to^{\Phi_3} < L_\to^{\Phi_2}$$
where $L < L'$ means that there is an $\{\to\}$-morphism from $L$ to $L'$ but $L$ and $L'$ are not equipollent.
Observe that $\mu_2$ is the classical conjunction and, more generally, we have that $\mu_n(\bar{x}) = \bigwedge_{i=1}^n s_i^n(\bar{x})$.

Example 3 of \cite{Caicedo04-ImpConnectivesAlgLogics} shows classical negation is implicitly definable in $L_\to$. Since none of the algebraic expansions of $L_\to$ has classical negation as a term, we have another example of an expansion by implicit connectives that is not algebraic.

\section{Algebraic expansions of Abelian $\ell$-groups and the Logic of Equilibrium}
\label{SECTION: l-groups}

In this section we give a complete description of the AE-classes of Abelian $\ell$-groups. In particular, we show that they form a lattice isomorphic with $\mathbf{1} \oplus \mathbf{2}^\omega$ (where $\oplus$ denotes the ordinal sum). In view of Theorem \ref{TEO: equivalencias interpretaciones} this produces a complete characterization of the algebraic expansions of the Logic of Equilibrium (\cite{GLS04,MOG05}).

Recall that an {\em Abelian $\ell$-group} is a structure in the language $\tau_\mathcal{G} := \{+,-,0,\vee,\wedge\}$ such that:
\begin{itemize} \setlength{\itemsep}{0pt}\setlength{\parskip}{0pt}
\item $(A,+,-,0)$ is an Abelian group,
\item $(A,\vee,\wedge)$ is a lattice,
\item $a + (b \vee c) = (a + b) \vee (a + c)$ for every $a,b,c \in A$.
\end{itemize}
Clearly Abelian $\ell$-groups form a variety, which we denote by $\mathcal{G}$. We write $\mathcal{G}_\mathrm{to}$ to denote its subclass of totally ordered members. Since all $\ell$-groups in this article are Abelian, we sometimes omit the word Abelian.

In the following lemma we collect some properties that are needed in the sequel.

\begin{lemma} \label{LEMA: props de l-grupos}
\
\begin{enumerate}
\item The variety $\mathcal{G}$ is arithmetical, that is, every member of $\mathcal{G}$ has permutable and distributive congruences.
\item \label{ITEM: minimal como univ} For every nontrivial $\mathbf{A} \in \mathcal{G}_\mathrm{to}$ we have $\mathsf{ISP_U}(\mathbf{A}) = \mathcal{G}_\mathrm{to}$.
\item \label{ITEM: fsi = to en G} An Abelian $\ell$-group is finitely subdirectly irreducible if and only if it is nontrivial and totally ordered.
\item \label{ITEM: minimal como quasi} For every nontrivial $\mathbf{A} \in \mathcal{G}$ we have $\mathsf{Q}(\mathbf{A}) = \mathsf{V}(\mathbf{A}) = \mathcal{G}$.
\end{enumerate}
\end{lemma}

\begin{proof}
Since $\ell$-groups have both group and lattice reducts, it is clear that they have permutable and distributive congruences (see \cite[Section II.12]{BurSan81-Book-CourseUnivAlg}). From the proof of Theorem 4 in \cite{Weinberg65-FreeLatOrdAbelGroupsII}, it follows that every finitely generated totally ordered Abelian $\ell$-group is embeddable in an ultrapower of $\mathbf{Z}$, the $\ell$-group of integers. This implies that $\mathcal{G}_\mathrm{to} = \mathsf{ISP_U}(\mathbf{Z})$, so 2. follows from the fact that $\mathbf{Z}$ is a subalgebra of any nontrivial $\ell$-group. Item 3. is proved in \cite[Lemma 3.5.4]{Glass99-Book-PartiallyOrderedGroups}. We prove 4.; fix a nontrivial $\mathbf{A}$ in $\mathcal{G}$ and note that $\mathsf{Q}(\mathbf{A}) \subseteq \mathsf{V}(\mathbf{A}) \subseteq \mathcal{G}$. Since $\mathbf{Z}$ is a substructure of $\mathbf{A}$, we have that $\mathcal{G}_\mathrm{to} = \mathsf{ISP_U}(\mathbf{Z}) \subseteq \mathsf{Q}(\mathbf{A})$. So $\mathsf{ISP}(\mathcal{G}_\mathrm{to}) \subseteq \mathsf{ISP}(\mathsf{Q}(\mathbf{A})) = \mathsf{Q}(\mathbf{A})$ and, as 3. says that $\mathcal{G} = \mathsf{ISP}(\mathcal{G}_\mathrm{to})$, we are done.
\end{proof}

\subsection{AE-classes of Abelian $\ell$-groups}

We proceed to characterize EFD-sentences modulo equivalence in $\mathcal{G}$. We first reduce the problem to totally ordered Abelian $\ell$-groups and then show the special role that divisible groups play as regards EFD-sentences.

\begin{lemma} \label{LEMA: equiv en Gto -> equiv en G}
Given EFD-sentences $\varphi, \psi$, if $\varphi \eq \psi$ in $\mathcal{G}_\mathrm{to}$, then $\varphi \eq \psi$ in $\mathcal{G}$.
\end{lemma}

\begin{proof}
Suppose $\varphi \eq \psi$ in $\mathcal{G}_\mathrm{to}$; take a nontrivial $\mathbf{A}$ in $\mathcal{G}$, and assume $\mathbf{A} \vDash \varphi$. On the one hand, since $U(\varphi)$ is a quasi-identity, Lemma \ref{LEMA: props de l-grupos}.(\ref{ITEM: minimal como quasi}) implies that  $\mathsf{H}(\mathbf{A}) \vDash U(\varphi)$. On the other hand, $\mathsf{H}(\mathbf{A}) \vDash E(\varphi)$ because $E(\varphi)$ is preserved by homomorphic images. Hence $\mathsf{H}(\mathbf{A}) \vDash \varphi$ and, in particular, $\mathsf{H}(\mathbf{A})_\mathrm{fsi} \vDash \varphi$. As, by Lemma \ref{LEMA: props de l-grupos}.(\ref{ITEM: fsi = to en G}), every member in $\mathsf{H}(\mathbf{A})_\mathrm{fsi}$ is totally ordered, we have $\mathsf{H}(\mathbf{A})_\mathrm{fsi} \vDash \psi$. So, using Lemma \ref{LEMA: aritm + fsi univ}, we are done.
\end{proof}

For each positive integer $k$ define \label{DEF: delta_k} $$\delta_k := \forall x \exists! z \, kz = x.$$ Our next step is to show that every EFD-sentence is equivalent to a $\delta_k$ in $\mathcal{G}$, which is accomplished in Theorem \ref{TEO: EFD in l-groups}.

Recall that an $\ell$-group $\mathbf{G}$ is {\em divisible} if for every $g \in G$ and every positive integer $n$, there exists $h \in G$ such that $g = nh$. Given a divisible $\ell$-group $\mathbf{D}$, since $\ell$-groups are torsion-free, we have that $\delta_k$ holds $\mathbf{D}$ for all $k$; thus, we can define the expansion $$\overline{\mathbf{D}} := (\mathbf{D}, ([\delta_k]^\mathbf{D})_{k\geq 1}).$$

The next result shows that the only functions defined by EFD-sentences in these expansions are term-operations.

\begin{theorem} \label{TEO: alg en divisibles son terminos con divisiones}
Let $\mathbf{D}$ be a totally ordered divisible $\ell$-group and let $\varphi$ be an EFD-sentence that holds in $\mathbf{D}$. Then, the functions $[\varphi]_1^\mathbf{D}, \ldots, [\varphi]_m^\mathbf{D}$ defined by $\varphi$ on $\mathbf{D}$ are term-functions on $\overline{\mathbf{D}}$.
\end{theorem}

The above theorem can be derived from \cite[Theorem 20]{Caicedo07-Implicit-in-Luka}. We provide a different proof that relies on the characterization of existentially closed algebras in $\mathcal{G}_\mathrm{to}$.

Given a class $\mathcal{K}$ of algebras closed under isomorphisms and $\mathbf{A} \in \mathcal{K}$, we say that $\mathbf{A}$ is {\em existentially closed} in $\mathcal{K}$ if for every $\mathbf{B} \in \mathcal{K}$ such that $\mathbf{A} \subseteq \mathbf{B}$, every existential formula $\varphi(\bar{x})$, and every $\bar{a} \in A^n$ $$\mathbf{B} \vDash \varphi(\bar{a}) \text{ implies } \mathbf{A} \vDash \varphi(\bar{a}).$$ The next proposition characterizes the existentially closed members of the class of totally ordered $\ell$-groups. 

\begin{proposition}
Given a totally ordered $\ell$-group $\mathbf{G}$, we have that $\mathbf{G}$ is existentially closed in $\mathcal{G}_\mathrm{to}$ if and only if $\mathbf{G}$ is divisible.
\end{proposition}

\begin{proof}
The result follows from \cite[Theorem 3.1.13]{Robinson56-Book-CompleteTheories} when considering totally ordered $\ell$-groups as structures on a purely relational language $\tau$. However, since the operations $+, -, \vee, \wedge$ are definable by quantifier-free $\tau$-formulas, the statement follows.
%
\end{proof}

\begin{corollary} \label{COR: divisible < t.o.}
Let $\mathbf{D} \subseteq \mathbf{G}$ be totally ordered $\ell$-groups and assume $\mathbf{D}$ is divisible. Then, for every EFD-sentence $\varphi$ we have that $\mathbf{G} \vDash \varphi$ implies $\mathbf{D} \vDash \varphi$.
\end{corollary}

\begin{proof}
Suppose $\mathbf{G}$ satisfies the EFD-sentence $\varphi$. Since $U(\varphi)$ is universal, we have $\mathbf{D} \vDash \varphi$, and the fact that $\mathbf{D}$ is existentially closed implies $\mathbf{D} \vDash E(\varphi)$.
\end{proof}

\begin{corollary} \label{COR: los divisibles t.o. satisfacen todo lo posible}
If $\varphi$ is an EFD-sentence with a nontrivial model in $\mathcal{G}$, then every totally ordered divisible $\ell$-group satisfies $\varphi$.
\end{corollary}

\begin{proof}
Assume $\mathbf{H}$ is a nontrivial model of $\varphi$ and let $\mathbf{H}'$ be a nontrivial totally ordered homomorphic image of $\mathbf{H}$. Clearly $\mathbf{H}' \vDash E(\varphi)$ and, since $\mathsf{Q}(\mathbf{H})$ is the class of all $\ell$-groups, we have $\mathbf{H}' \vDash U(\varphi)$. Hence $\mathbf{H}' \vDash \varphi$.
By Lemma \ref{LEMA: props de l-grupos}, we know that $\mathsf{ISP_U}(\mathbf{H}') = \mathcal{G}_\mathrm{to}$. Thus, if $\mathbf{D}$ is a totally ordered divisible $\ell$-group, there is $\mathbf{G} \in \mathsf{P_U}(\mathbf{H}')$ such that $\mathbf{D} \subseteq \mathbf{G}$. Finally, Corollary \ref{COR: divisible < t.o.} yields $\mathbf{D} \vDash \varphi$.
\end{proof}

After this sequence of results we are ready to present:

\begin{proof}[Proof of Theorem \ref{TEO: alg en divisibles son terminos con divisiones}]
Assume $\mathbf{D} \vDash \varphi$ for some EFD-sentence $\varphi$, $\mathbf{D}$ nontrivial. Observe that $\mathsf{V}(\overline{\mathbf{D}})$ is arithmetical since arithmeticity is witnessed by a Pixley term (see \cite{BurSan81-Book-CourseUnivAlg}).

We prove first that $\mathsf{V}(\overline{\mathbf{D}})_\mathrm{fsi} \vDash \varphi$. Since all divisions are basic operations of $\overline{\mathbf{D}}$, we have that the algebras in $\mathsf{SP_U}(\overline{\mathbf{D}})$ are totally ordered divisible $\ell$-groups, and Corollary \ref{COR: los divisibles t.o. satisfacen todo lo posible} produces $\mathsf{SP_U}(\overline{\mathbf{D}}) \vDash \varphi$. Clearly $\mathsf{HSP_U}(\overline{\mathbf{D}}) \vDash E(\varphi)$ and, since $\mathcal{G} = \mathsf{Q}(\mathbf{D}) \vDash U(\varphi)$, it follows that $\mathsf{HSP_U}(\overline{\mathbf{D}})$ satisfies $U(\varphi)$ as well. Thus, $\mathsf{HSP_U}(\overline{\mathbf{D}}) \vDash \varphi$, and we are done since $\mathsf{V}(\overline{\mathbf{D}})_\mathrm{fsi} \subseteq \mathsf{HSP_U}(\overline{\mathbf{D}})$ by Jónsson's lemma (see \cite{Jonsson67-AlgCongLatDistributive}).

Since $\ell$-group congruences are compatible with division operations, we have that the congruences of algebras in $\mathsf{V}(\overline{\mathbf{D}})$ agree with the congruences of their $\ell$-group reducts, and so, $\mathsf{V}(\overline{\mathbf{D}})_\mathrm{fsi}$ is a universal class. Thus, by Lemma \ref{LEMA: aritm + fsi univ}, $\mathsf{V}(\overline{\mathbf{D}}) \vDash \varphi$, and the conclusion follows now from \cite[Lemma 3]{CamVag11-AlgebraicFunctions}.
\end{proof}

Given a positive integer $k$ and a term $t(\bar{x})$ in $\tau_\mathcal{G}$, let $$\delta_{k,t} := \forall \bar{x} \exists ! z \, kz = t(\bar{x}).$$ Observe that $U(\delta_{k,t})$ is valid in $\mathcal{G}$ because Abelian $\ell$-groups are torsion-free.

We denote by $\mathcal{D}$ the class of expansions $\overline{\mathbf{D}}$ of divisible $\ell$-groups $\mathbf{D} \in \mathcal{G}$.  We write $\tau_{\mathcal{D}}$ for the language of the algebras in the class $\mathcal{D}$.

\begin{lemma} \label{LEMA: equivalencia con fi k t}
Given a term $s(\bar{x})$ in $\tau_{\mathcal{D}}$, there is a term $t(\bar{x})$ in $\tau_\mathcal{G}$ and a positive integer $k$ such that $k \, s(\bar{x}) = t(\bar{x})$ is valid in $\mathcal{D}$. Hence, for any divisible $\mathbf{D} \in \mathcal{G}$ the term-function $s^{\overline{\mathbf{D}}}$ agrees with the function $[\delta_{k,t}]^\mathbf{D}$.
\end{lemma}

\begin{proof}
It follows by induction on the structure of $s(\bar{x})$.
\end{proof}

\begin{lemma} \label{LEMA: reduccion a phi k t}
Let $\varphi$ be an EFD-sentence with a nontrivial model in $\mathcal{G}$. Then there are positive integers $k_1, \dots, k_m$ and terms $t_1, \dots, t_m$ in $\tau_\mathcal{G}$ such that $\varphi \eq \displaystyle \bigand_{j=1}^m \delta_{k_j,t_j}$ in $\mathcal{G}$.
\end{lemma}

\begin{proof}
Fix $\varphi := \forall x_1 \dots x_n \exists! z_1 \dots z_m \, \alpha(\bar{x},\bar{z})$. Note that $\mathcal{G} \vDash U(\varphi)$ since $\varphi$ has a nontrivial model in $\mathcal{G}$ and $\mathcal{G}$ has no proper subquasivarieties. Let $\mathbf{D}$ be a nontrivial totally ordered divisible $\ell$-group. By Corollary \ref{COR: los divisibles t.o. satisfacen todo lo posible}, we have that $\mathbf{D} \vDash \varphi$. So Theorem \ref{TEO: alg en divisibles son terminos con divisiones} provides terms $s_1, \dots, s_m$ in $\tau_{\mathcal{D}}$ such that $[\varphi]_j^\mathbf{D} = s_j^{\overline{\mathbf{D}}}$ for $j \in \{1,\dots,m\}$. Moreover, by Lemma \ref{LEMA: equivalencia con fi k t}, there are positive integers $k_1,\dots,k_m$ and terms $t_1, \dots, t_m$ in $\tau_\mathcal{G}$ such that $s_j^{\overline{\mathbf{D}}} = [\delta_{k_j,t_j}]^\mathbf{D}$. This shows that $\mathbf{D} \vDash \forall \bar{x} \bar{z} \, ( \alpha(\bar{x},\bar{z}) \biff \displaystyle \bigand_{j=1}^m k_jz_j = t_j(\bar{x}) )$, and again using that $\mathcal{G}$ has no proper subquasivarieties, we have $\mathcal{G} \vDash \forall \bar{x} \bar{z} \, (\alpha(\bar{x},\bar{z}) \biff \displaystyle \bigand_{j=1}^m k_jz_j = t_j(\bar{x}))$. Finally, since $\mathcal{G}$ satisfies $U(\varphi)$ and $U(\delta_{k_j,t_j})$ for $j \in \{1,\dots,m\}$, it follows that $\varphi \eq \displaystyle \bigand_{j=1}^m \delta_{k_j,t_j}$ in $\mathcal{G}$.
\end{proof}

In the following, by a {\em system of linear inequalities} we mean a finite conjunction of inequalities of the form $a_1x_1 + \dots + a_nx_n \geq 0$ where $a_1,\dots,a_n$ are integers. (Note that such a system can be written as a conjunction of equations in $\tau_\mathcal{G}$.)

We say that a system of linear inequalities $\alpha(\bar{x})$ is {\em full-dimensional} on an Abelian $\ell$-group $\mathbf{G}$ if there is no $(a_1,\dots,a_n) \in \mathbb{Z}^n \setminus \{\bar{0}\}$ such that $\mathbf{G} \vDash \forall \bar{x} \, (\alpha(\bar{x}) \bimp \sum a_ix_i = 0)$. That is, the system $\alpha(\bar{x})$ imposes no linear dependencies on its solutions in $\mathbf{G}$. Observe that Lemma \ref{LEMA: props de l-grupos}.(\ref{ITEM: minimal como quasi}) implies that $\alpha(\bar{x})$ is full-dimensional on some nontrivial $\ell$-group $\mathbf{G}$ if and only if it is full-dimensional on every nontrivial $\ell$-group. Hence, we say that $\alpha(\bar{x})$ is {\em full-dimensional} provided it is full-dimensional on some nontrivial $\ell$-group.

\begin{lemma} \label{LEMA: independiente <-> genera}
A system of inequalities $\alpha(\bar{x})$ is full-dimensional if and only if for every totally ordered $\ell$-group $\mathbf{G}$ the set $\{\bar{g} \in G^n: \mathbf{G} \vDash \alpha(\bar{g})\}$ generates $\mathbf{G}^n$ as an Abelian group.
\end{lemma}

\begin{proof}
Assume $\alpha(\bar{x})$ is a full-dimensional system of inequalities and let $S_\mathbf{G} := \{\bar{g} \in G^n: \mathbf{G} \vDash \alpha(\bar{g})\}$ for any totally ordered $\ell$-group $\mathbf{G}$. Let $\mathbf{Q}$ and $\mathbf{Z}$ denote the $\ell$-groups of rational and integer numbers, respectively. First observe that $S_\mathbf{Z} = S_\mathbf{Q} \cap Z^n$. Note also that $S_\mathbf{G}$ is closed under linear combinations whose coefficients are non-negative integers, and $S_\mathbf{Q}$ is closed under non-negative rational linear combinations.

We start by proving that $S_\mathbf{Z}$ generated $\mathbf{Z}^n$ as Abelian group. Let $V$ be the $\mathbf{Q}$-vector subspace of $\mathbf{Q}^n$ generated by $S_\mathbf{Q}$. Observe that $V = Q^n$; otherwise, there would exist integers $a_1, \ldots, a_n$, not all zero, such that $V \subseteq \{\bar{x} \in Q^n: \sum a_ix_i = 0\}$, contradicting the fact that $\alpha(\bar{x})$ is full-dimensional. Since $V = Q^n$, the solution set $S_\mathbf{Q}$ contains a $\mathbf{Q}$-basis of $\mathbf{Q}^n$, which, multiplied by a suitable positive integer, yields a $\mathbf{Q}$-basis $\{\bar{b}_1,\dots,\bar{b}_n\} \subseteq S_\mathbf{Z}$.  Since $S_\mathbf{Z}$ is closed under positive integer linear combinations, $\bar{b} := \sum \bar{b}_i \in S_\mathbf{Z}$. Now, let $\bar{c} \in Z^n$ be arbitrary and write $\bar{c}  = \sum r_i\bar{b}_i$ for suitable rational numbers $r_i$. Let $k$ be a positive integer such that $k \geq -r_i$ for all $i$. Then $k \bar{b} + \bar{c} = \sum_i (k+r_i)\bar{b}_i \in S_\mathbf{Q}$, since it is a positive linear combination of elements in $S_\mathbf{Q}$. Thus $k\bar{b} + \bar{c} \in S_\mathbf{Q} \cap Z^n = S_\mathbf{Z}$, and so $\bar{c} = (k\bar{b}+\bar{c}) - k\bar{b}$ belongs to the Abelian group generated by $S_\mathbf{Z}$.

We prove now that $S_\mathbf{G}$ generates $\mathbf{G}^n$ as an Abelian group for any totally ordered group $\mathbf{G}$. For any $\bar{a} \in Z^n$ and $g \in G$ we write $\bar{a}g := (a_1g, \dots, a_ng)$. Note that if $\bar{a} \in S_\mathbf{Z}$ and $g$ is a non-negative member of $G$, then $\bar{a}g \in S_\mathbf{G}$. Fix $j \in \{1,\ldots,n\}$, and let $\bar{e}_j \in Z^n$ be such that $e_{ji} = 1$ if $i = j$ and $e_{ji} = 0$ otherwise. We write $\bar{e}_j = \sum k_l\bar{a}_l$ for integers $k_l$ and $\bar{a}_l \in S_\mathbf{Z}$. Hence, if $g \in G$, $g \geq 0$, then $\bar{e}_jg = \sum k_l \bar{a}_lg$ is an integer linear combination of solutions $\bar{a}_lg \in S_\mathbf{G}$. This proves that $S_\mathbf{G}$ generates $\bar{e}_jg$ for all $j$ and all $g \in G$, $g \geq 0$. Now it follows easily that any $\bar{g} \in G^n$ is generated by elements in $S_\mathbf{G}$.

The converse implication is straightforward.
\end{proof}

\begin{lemma} \label{LEMA: forma canonica de terminos}
Let $t(\bar{x})$ be a term in $\tau_\mathcal{G}$. There are full-dimensional systems of linear inequalities $\alpha_1(\bar{x}), \dots, \alpha_m(\bar{x})$ and terms $t_1(\bar{x}), \dots, t_m(\bar{x})$, which are integer linear combinations of the variables $x_1, \ldots, x_n$, such that for all $\mathbf{G} \in \mathcal{G}_\mathrm{to}$ and all $\bar{g} \in G^n$ we have
\begin{equation}
t^\mathbf{G}(\bar{g}) = \begin{cases} t_1^\mathbf{G}(\bar{g}) & \text{if } \alpha_1(\bar{g}), \\ \;\; \vdots \\ t_m^\mathbf{G}(\bar{g}) & \text{if }\alpha_m(\bar{g}). \end{cases}
\label{EQ: def por casos en G}
\end{equation}
\end{lemma}

\begin{proof}
Fix a $\tau_\mathcal{G}$-term $t(\bar{x})$. We show first that there are full-dimensional systems $\alpha_1(\bar{x}), \dots, \alpha_m(\bar{x})$ and Abelian group terms $t_1(\bar{x}), \dots, t_m(\bar{x})$ such that \eqref{EQ: def por casos en G} holds for $\mathbf{G} = \mathbf{R}$, the $\ell$-group of real numbers.

Using the way the lattice and group operations interact, we may assume $t(\bar{x}) = s(u_1(\bar{x}),\dots,u_p(\bar{x}))$ where $u_1(\bar{x}),\dots,u_p(\bar{x})$ are Abelian group terms (i.e., linear combinations of variables with integer coefficients) and $s(\bar{y})$ is a lattice term. For each permutation $\sigma$ of $\{1,\ldots,p\}$ let $\alpha_\sigma(\bar{x})$ be the system of linear inequalities expressing that $u_{\sigma(1)}(\bar{x}) \leq \dots \leq u_{\sigma(p)}(\bar{x})$. Since $\mathbf{R}$ is totally ordered, for each $\sigma$ there is $j_\sigma \in \{1,\ldots,p\}$ such that
$$t^\mathbf{R}(\bar{r}) = u_{j_\sigma}^\mathbf{R}(\bar{r}) \text{ for all } \bar{r} \text{ such that } \alpha_\sigma(\bar{r}).$$

Next, for each $\sigma$ let $S_\sigma := \{\bar{r} \in R^n: \alpha_\sigma(\bar{r})\}$. As each $\bar{r} \in R^n$ satisfies at least one $\alpha_\sigma(\bar{x})$, we have that $R^n = \bigcup_\sigma S_\sigma$. Let $\{\sigma_1,\dots,\sigma_m\}$ be the set of permutations $\sigma$ such that $S_\sigma$ has nonempty interior. Note that $\alpha_{\sigma_j}(\bar{x})$ is full-dimensional on $\mathbf{R}$ for all $j \in \{1,\ldots,m\}$ (and thus on every $\ell$-groups).
Since each $S_\sigma$ is a closed subset of $R^n$, by a simple topological argument, we have
$$R^n = S_{\sigma_1} \cup \dots \cup S_{\sigma_m}.$$
So, defining $\alpha_j(\bar{x}) := \alpha_{\sigma_j}(\bar{x})$ and $t_j(\bar{x}) := u_{\sigma_j}(\bar{x})$ for $j \in \{1,\dots,m\}$, we have established \eqref{EQ: def por casos en G} in the case $\mathbf{G} = \mathbf{R}$. To conclude we show that the same $\alpha_j$'s and $t_j$'s work for any $\mathbf{G} \in \mathcal{G}_\mathrm{to}$. In fact, note that \eqref{EQ: def por casos en G} holds if and only if $\mathbf{G}$ satisfies the following universal formulas
\begin{itemize} \setlength{\itemsep}{0pt}\setlength{\parskip}{0pt}
\item $\forall \bar{x} \, (\alpha_j(\bar{x}) \bimp t(\bar{x}) = t_j(\bar{x}))$ for $j \in \{1,\dots,m\}$,
\item $\forall \bar{x} \, (\alpha_1(\bar{x}) \bor \dots \bor \alpha_m(\bar{x}))$.
\end{itemize}
Since these formulas hold in $\mathbf{R}$, Lemma \ref{LEMA: props de l-grupos}.(\ref{ITEM: minimal como univ}) says that they must hold in $\mathbf{G}$.
\end{proof}

\begin{lemma} \label{LEMA: reduccion a phi k}
Given a positive integer $k$ and an $\tau_\mathcal{G}$-term $t$, there is a positive integer $k'$ such that $\delta_{k,t} \eq \delta_{k'}$ in $\mathcal{G}$.
\end{lemma}

\begin{proof}
Fix a positive integer $k$ and an $\tau_\mathcal{G}$-term $t$; let $\alpha_j(\bar{x})$ and $t_j(\bar{x})$ for $j \in \{1,\ldots,m\}$ be as in Lemma \ref{LEMA: forma canonica de terminos}. Suppose $t_j(\bar{x}) = a_{j1}x_1 + \dots + a_{jn}x_n$, and let $d$ be the greatest common divisor of the set $\{k\} \cup \{a_{ji}: i \in \{1,\dots,n\}, j \in \{1,\ldots,m\}\}$. Define $k'$ by $k = dk'$; we prove that $\delta_{k,t} \eq \delta_{k'}$ in $\mathcal{G}$. Observe that, due to Lemma \ref{LEMA: equiv en Gto -> equiv en G}, it suffices to show that $\delta_{k,t} \eq \delta_{k'}$ in $\mathcal{G}_\mathrm{to}$.

Take $\mathbf{G} \in \mathcal{G}_\mathrm{to}$ and assume $\mathbf{G} \vDash \delta_{k,t}$. We claim that $t_j(\bar{g})$ is divisible by $k$ for every $\bar{g} \in G^n$ and $j \in \{1,\dots,m\}$. Indeed, given $\bar{g} \in G^n$ and $j \in \{1,\dots,m\}$, by Lemma \ref{LEMA: independiente <-> genera}, we can write $\bar{g} = \sum b_l \bar{g}_l$ for some integers $b_l$ and some $\bar{g}_l \in G^n$ such that $\mathbf{G} \vDash \alpha_j(\bar{g}_l)$. Note that $t(\bar{g}_l) = t_j(\bar{g}_l)$ for each $l$. Since $\mathbf{G} \vDash \delta_{k,t}$, for each $l$ there is $h_l \in G$ such that $k h_l = t(\bar{g}_l) = t_j(\bar{g}_l)$. Thus
\begin{align*}
t_j(\bar{g}) & = t_j(\sum b_l \bar{g}_l) \\
& = \sum b_l t_j(\bar{g}_l) \\
& = \sum b_l k h_l \\ 
& = k \sum b_lh_l,
\end{align*}
which proves the claim.

Now write $d = kc + \sum_{i,j} a_{ji}c_{ji}$ for suitable integers $c$ and $c_{ji}$. Then, for any $g \in G$,
\begin{align*}
dg & = kcg + \sum_j \sum_i a_{ji}c_{ji}g \\
& = kcg + \sum_j t_j(\bar{g}_j)
\end{align*}
where $\bar{g}_j := (c_{j1}g,\dots,c_{jn}g)$. Since each $t_j(\bar{g}_j)$ is divisible by $k$, it follows that there is $g' \in G$ such that $dg = kg'$. Thus $dg = dk'g'$, so $d(g-k'g') = 0$ and, since $\mathbf{G}$ is torsion-free, $g = k'g'$. This proves that $\mathbf{G} \vDash \delta_{k'}$.

Conversely, assume any element in $\mathbf{G}$ is divisible by $k'$, and fix $\bar{g} := (g_1,\dots,g_n) \in G^n$. We prove that $t(\bar{g})$ is divisible by $k$. Let $j \in \{1,\dots,m\}$ be such that $t(\bar{g}) = t_j(\bar{g})$. Since each $a_{ji}$ is divisible by $d$, there is $g' \in G$ such that $t_j(\bar{g}) = d g'$. Now, since $g'$ is divisible by $k'$, there is $g'' \in G$ such that $g' = k'g''$. Putting all together we obtain $t(\bar{g}) = t_j(\bar{g}) = dg' = dk'g'' = kg''$. 
\end{proof}

We are now ready to present our characterization of EFD-sentences in $\mathcal{G}$. 

\begin{theorem} \label{TEO: EFD in l-groups}
Given an EFD-sentence $\varphi$ with a nontrivial model in $\mathcal{G}$ there is a positive integer $k$ such that $\varphi \eq \delta_k$ in $\mathcal{G}$.
\end{theorem}

\begin{proof}
Given $\varphi$, combining Lemmas \ref{LEMA: reduccion a phi k t} and \ref{LEMA: reduccion a phi k}, we have that there are positive integers $k_1,\dots,k_m$ such that $\varphi \eq \displaystyle \bigand_{j=1}^m \delta_{k_j}$ in $\mathcal{G}$. Now take $k := k_1 \cdots k_m$, and note that $\displaystyle \bigand_{j=1}^m \delta_{k_j}$ is equivalent to $\delta_k$  in $\mathcal{G}$.   
\end{proof}

Given a set $S$ of prime numbers, let $\Sigma_S := \{\delta_p: p \in S\}$. Since for an $\ell$-group divisibility by $k$ is equivalent to divisibility by the prime factors of $k$, we have the following:

\begin{theorem} \label{TEO: clases AE de l-grupos}
Every set of EFD-sentences either has only trivial models or is equivalent over $\mathcal{G}$ to $\Sigma_S$ for some set $S$ of prime numbers. Furthermore, the map $S \mapsto \Sigma_S$ is one-to-one, and thus, the lattice of AE-subclasses of $\mathcal{G}$ is isomorphic with $\mathbf{1} \oplus \mathbf{2}^\omega$.
\end{theorem}

\subsection{The algebraic expansions of the Logic of Equilibrium}

As shown in \cite{GLS04} the variety $\mathcal{G}$ of Abelian $\ell$-groups is the equivalent algebraic semantics of the Logic of Equilibrium $Bal$ defined by the following:
\begin{multicols}{2}
Axioms

$(\varphi \imp \psi) \imp ((\theta \imp \varphi) \imp (\theta \imp \psi))$

$(\varphi \imp (\psi \imp \theta)) \imp (\psi \imp (\varphi \imp \theta))$

$((\varphi \imp \psi) \imp \psi) \imp \varphi$

$((\psi \imp \varphi)^+ \imp (\varphi \imp \psi)^+) \imp (\varphi \imp \psi)$

$\varphi^{++} \imp \varphi^+$

\columnbreak

Rules

$\varphi, \varphi \imp \psi \vdash \psi$

$\varphi, \psi \vdash \varphi \imp \psi$

$\varphi \vdash \varphi^+$

$(\varphi \imp \psi)^+ \vdash (\varphi^+ \imp \psi^+)^+$
\end{multicols}
\noindent The derived connectives:
\begin{itemize}
\item[] $0 := x \imp x$,
\item[] $-x := x \imp 0$,
\item[] $x + y := -x \imp y$,
\item[] $x \vee y := (x \imp y)^+ + x$,
\item[] $x \wedge y := -(-x \vee -y)$.
\end{itemize}
form a complete set for $Bal$ since $x \imp y \dashv\vdash -x + y$ and $x^+ \dashv\vdash x \vee 0$. This allows us to say that $\mathcal{G}$ is the equivalent algebraic semantics of $Bal$ via equivalence formulas $\Delta(x,y) = \{x \imp y\}$ and defining equations $\varepsilon(x) = \{x = 0\}$.

Given a prime number $p$, the algebraic expansion of $Bal$ corresponding to the EFD-sentence $\delta_p$ is, by definition, obtained from $Bal$ by adding the axiom:
\begin{equation}
x \imp p\mathsf{d}_p(x) \tag{$A_p$}
\end{equation}
and the rule $\mathrm{U}_{\{x \imp pz\}}$. Since this rule is derivable in $Bal$, the expansion is obtained simply by adding $A_p$. For a set $S$ of prime numbers define $Bal^S$ as the expansion of $Bal$ by the axioms $\{A_p: p \in S\}$. Note that, since $Bal^S$ is an axiomatic expansion of $Bal$, its equivalent algebraic semantics is a variety. These expansions were also considered in \cite{Caicedo07-Implicit-in-Luka} where it is proved that every implicit connective in the logic $Bal^{Primes}$ is explicit.

Recall that an expansion $L'$ of a logic $L$ is called {\em conservative} provided that for each set of $L$-formulas $\Gamma \cup \{\varphi\}$ we have that $\Gamma \vdash_{L'} \varphi$ implies $\Gamma \vdash_L \varphi$.

\begin{theorem} \label{TEO: exp alg de Bal}
\
\begin{enumerate}
\item Every algebraic expansion of $Bal$ is $\tau_\mathcal{G}$-equipollent to exactly one of the following:
\begin{itemize}
\item Inconsistent Logic,
\item $Bal^S$ for some set $S$ of prime numbers.
\end{itemize}
\item The algebraic expansions of $Bal$ form a lattice isomorphic with $\mathbf{2}^\omega \oplus \mathbf{1}$ when ordered by $\tau_\mathcal{G}$-morphisms.
\item Given sets $S,S'$ of prime numbers with $S \subseteq S'$, the expansion $Bal^{S'}$ is conservative over $Bal^S$.
\end{enumerate}
\end{theorem}

\begin{proof}
Items 1. and 2. follow from Theorems \ref{TEO: equivalencias interpretaciones} and \ref{TEO: clases AE de l-grupos}. We prove 3.

Fix sets of prime numbers $S \subseteq S'$, and let $\mathcal{V}$ and $\mathcal{V}'$ be the equivalent algebraic semantics of $Bal^S$ and $Bal^{S'}$, respectively. Since $Bal^{S'}$ is finitary, to prove 3. it is enough to show that any quasi-identity in the language of $\mathcal{V}$ valid in $\mathcal{V}'$ is also valid in $\mathcal{V}$. Let $\mathbf{Q}_S$ be the $\ell$-group of rational numbers expanded with  the divisions by the primes in $S$. It is not hard to show that $\mathbf{Q}_S$ generates $\mathcal{V}$ as a quasivariety, that is, $\mathsf{Q}(\mathbf{Q}_S) = \mathcal{V}$. Now let $\varphi$ be a quasi-identity in the language of $\mathcal{V}$ that is valid in $\mathcal{V}'$. Then, we have that $\mathbf{Q}_{S'} \vDash \varphi$, and thus, $\mathbf{Q}_S \vDash \varphi$. Since $\mathsf{Q}(\mathbf{Q}_S) = \mathcal{V}$, the proof is finished.
\end{proof}

\section{Algebraic expansions of perfect MV-algebras and their logic}
\label{SECTION: hoops and MV}

In this section we characterize the AE-subclasses of the variety $\mathsf{V}(\mathcal{P})$, where $\mathcal{P}$ is the class of perfect MV-algebras. So we also obtain a full description of the algebraic expansions of $L_\mathcal{P}$, the Logic of Perfect MV-Algebras (see, e.g., \cite{BelDiNGer07-PerfectMVAlg}).

Our approach is to export the results for Abelian $\ell$-groups to perfect MV-algebras, exploiting the connection between these two classes (see \cite{Mundici86-MappAbelLGroupsStrongUnitMVAlg,Mundici86-InterpAFinLuk,DiNLet94-PerfectMVCategEquivToAbelLGroups}). 
We first translate the classification of EFD-sentences given in Theorem \ref{TEO: EFD in l-groups} to cancellative hoops ---the positive cones of Abelian $\ell$-groups. These structures provide a stepping stone to carry our results over to perfect MV-algebras.

\subsection{EFD-sentences on cancellative hoops}

Given an Abelian $\ell$-group $\mathbf{G}$, its {\em positive cone} is the subset $G^+ := \{x \in G: x \geq 0\}$. We define the algebraic structure $\mathbf{G}^+ := (G^+, +, \dotdiv, 0)$ where $x \dotdiv y := (x-y) \vee 0$. The lattice structure of $G^+$ (as a sublattice of $\mathbf{G}$) is determined by the operations $+$ and $\dotdiv$. Indeed, for every $x,y \in G^+$ we have that $x \vee y = x + (y \dotdiv x)$ and $x \wedge y = x \dotdiv (x \dotdiv y)$.

We write $\mathcal{H}$ for the class of positive cones of Abelian $\ell$-groups considered as algebras in the language $\tau_\mathcal{H} := \{+, \dotdiv, 0\}$. The members of $\mathcal{H}$ are known as {\em cancellative hoops}; see e.g. \cite{Ferreirim92-Thesis,BloFer00-OnStructHoops}. Given a cancellative hoop $\mathbf{A}$, there is (up to isomorphism) a unique Abelian $\ell$-group whose positive cone is isomorphic with $\mathbf{A}$ (see \cite{Birkhoff40-Book-LatticeTheory}); we write $\mathbf{A}$ for this $\ell$-group. Moreover, $\mathbf{A}$ is totally ordered if and only if $\mathbf{A}$ is. We write $\mathcal{H}_\mathrm{to}$ for the class of totally ordered cancellative hoops.


\begin{lemma} \label{LEMA: props de hoops cancelativos}
\
\begin{enumerate}
\item The class $\mathcal{H}$ is an arithmetical variety.
\item For every nontrivial $\mathbf{A} \in \mathcal{H}$, we have $\mathsf{Q}(\mathbf{A}) = \mathsf{V}(\mathbf{A}) = \mathcal{H}$.
\item A cancellative hoop is finitely subdirectly irreducible if and only if it is nontrivial and totally ordered.
\end{enumerate}
\end{lemma}

\begin{proof}
These results follow from the general theory of hoops. See \cite{Bosbach69-KomplHalbgruppen-AxiomArith,Bosbach70-KomplHalbgruppen-KongQuot,BloFer00-OnStructHoops,Ferreirim92-Thesis}.
%
%
%
%
%
\end{proof}

\begin{lemma}
Given EFD-sentences $\varphi, \psi$, if $\varphi \eq \psi$ in $\mathcal{H}_\mathrm{to}$, then $\varphi \eq \psi$ in $\mathcal{H}$.
\end{lemma}

\begin{proof}
The proof is analogous to the one for Lemma \ref{LEMA: equiv en Gto -> equiv en G} using Lemmas \ref{LEMA: aritm + fsi univ} and \ref{LEMA: props de hoops cancelativos}.
\end{proof}

Our next step is to provide a translation of EFD-sentences in $\tau_\mathcal{H}$ into EFD-sentences in $\tau_\mathcal{G}$. First, for an $\tau_\mathcal{H}$-term $t(\bar{x},\bar{z})$ define recursively the $\tau_\mathcal{G}$-term $t^{*}(\bar{x},\bar{z})$ by
\begin{itemize}
\item $0^{*} :=0$
\item $x_{i}^{*} := x_{i}\vee-x_{i}$
\item $z_{j}^{*} := z_{j}$
\item $(t_{1}+t_{2})^{*} := t_{1}^{*}+t_{2}^{*}$
\item $(t_{1}\dotdiv t_{2})^{*} := (t_{1}^{*}-t_{2}^{*})\vee0$.
\end{itemize}
Next, given a conjunction of equations $\alpha(\bar{x},\bar{z}) := \displaystyle \bigand_{i=1}^k t_i(\bar{x},\bar{z}) = s_i(\bar{x},\bar{z})$ in the language $\tau_\mathcal{H}$, define the $\tau_\mathcal{G}$-formula $$\alpha^*(\bar{x},\bar{z}) := \displaystyle \bigand_{i=1}^k t_i^*(\bar{x},\bar{z}) = s_i^*(\bar{x},\bar{z}).$$ Finally, if $\varphi := \forall \bar{x} \exists ! \bar{z} \, \alpha(\bar{x},\bar{z})$ is an EFD-sentence in the language $\tau_\mathcal{H}$, we define the translation of $\varphi$ into $\tau_\mathcal{G}$ as the sentence
\[
\varphi^{*} := \forall \bar{x} \exists! \bar{z} \, \alpha^{*}(\bar{x},\bar{z}) \band z_{1} \geq 0 \band \dots \band z_{m} \geq 0.
\]
The next lemma shows that our translations work as intended.

\begin{lemma} \label{LEMA: prop basicas de la estrella}
Let $\mathbf{A} \in \mathcal{H}_\mathrm{to}$, let $\alpha(\bar{x},\bar{z})$ be a conjunction of $\tau_\mathcal{H}$-equations, and let $\varphi$ be an EFD-sentence in the language $\tau_\mathcal{H}$.
\begin{enumerate}
\item For all $\bar{a},\bar{b}$ from $A$ we have $\mathbf{A} \vDash \alpha(\bar{a},\bar{b})$ if and only if $\mathbf{A}^* \vDash \alpha^*(\bar{a},\bar{b})$.
\item $\mathbf{A} \vDash \varphi$ if and only if $\mathbf{A}^* \vDash\varphi^*$.
\end{enumerate}
\end{lemma}

\begin{proof}
To prove 1. it suffices to show that $t^\mathbf{A}(\bar{a},\bar{b})=t^{*\mathbf{A}^*}(\bar{a},\bar{b})$ for all $\bar{a},\bar{b}$ from $A$, which is a easy induction on the structure of $t(\bar{x},\bar{z})$.

Next, we prove the left-to-right implication of 2. Assume $\varphi := \forall \bar{x} \exists! \bar{z} \, \alpha(\bar{x},\bar{z})$ is valid in $\mathbf{A}$. Given $c_1,\dots,c_n \in A^*$, for each $i$ let $a_i := c_i \vee -c_i$ (note that $a_i \in A$). There are $b_1,\dots,b_m \in A$ such that $\mathbf{A} \vDash \alpha(\bar{a},\bar{b})$; thus, by item 1., $\mathbf{A}^* \vDash \alpha^*(\bar{a},\bar{b})$. By the definition of $\alpha^*$, this is equivalent to $\mathbf{A}^* \vDash \alpha^*(\bar{c},\bar{b})$; hence $\mathbf{A}^* \vDash E(\varphi^*)$. To prove the uniqueness part suppose there are $c_1,\dots,c_n$, $b_1,\dots,b_m$, $b_1',\dots,b_m' \in A^*$ such that $\mathbf{A}^* \vDash \alpha^*(\bar{c},\bar{b}) \band \alpha^*(\bar{c},\bar{b}') $ and $b_i \geq 0$, $b_i' \geq 0$ for all $i$. If for each $i$ we take $a_i := c_i \vee -c_i$, then we have $\mathbf{A}^* \vDash \alpha^*(\bar{a},\bar{b}) \band \alpha^*(\bar{a},\bar{b}')$, so $\mathbf{A} \vDash \alpha(\bar{a},\bar{b}) \band \alpha(\bar{a},\bar{b}')$. Thus $\bar{b}=\bar{b}'$, and we have shown $\mathbf{A}^* \vDash \varphi^*$.

The right-to-left implication is straightforward and left to the reader.
\end{proof}

From Section \ref{SECTION: l-groups} recall that $\delta_{k}$ is the sentence $\forall x \exists! z \, kz = x$. Note that $\delta_{k}$
is an EFD-sentence in both $\tau_\mathcal{H}$ and $\tau_\mathcal{G}$.

\begin{lemma} \label{lem:def en l-grupos}
For all positive integers $k$ we have $\delta_k^* \eq \delta_{k}$ in $\mathcal{G}$.
\end{lemma}

\begin{proof}
By Lemma \ref{LEMA: equiv en Gto -> equiv en G}, it suffices to check the equivalence of $\delta_k$ and $\delta_k^*$ in $\mathcal{G}_\mathrm{to}$. Let $\mathbf{G} \in \mathcal{G}_\mathrm{to}$ be such that $\mathbf{G} \vDash \delta_k^*$, i.e., $\mathbf{G}$ satisfies
$$\forall x \exists! z \, (kz = x \vee -x \band z \geq 0).$$
Given $a \in G$, there is $b \geq 0$ in $G$ such that $kb = a \vee -a$. If $a \geq 0$, then $kb = a$. Otherwise, $a \leq 0$ and $kb = -a$, so $k(-b) = a$. This shows that $\mathbf{G} \vDash \delta_k$. Conversely, assume $\mathbf{G} \vDash \delta_k$. Given $a \in G$, there is $b \in G$ such that $kb = a \vee -a$. Since $kb \geq 0$, it follows that $b \geq 0$. Thus $\mathbf{G} \vDash \delta_k^*$.
\end{proof}

We are now ready to prove a characterization of EFD-sentences for $\mathcal{H}$ analogous to Theorem~\ref{TEO: EFD in l-groups}.

\begin{theorem} \label{TEO: phi <-> phi_k in H}
Given an EFD-sentence $\varphi$ with a nontrivial model in $\mathcal{H}$ there is a positive integer $k$
such that $\varphi \eq \delta_{k}$ in $\mathcal{H}$.
\end{theorem}

\begin{proof}
Assume $\varphi$ has a nontrivial model in $\mathcal{H}$. By Lemma \ref{LEMA: prop basicas de la estrella}.2., the sentence $\varphi^*$ has a nontrivial model in $\mathcal{G}$. Thus, by Theorem \ref{TEO: EFD in l-groups} and Lemma \ref{lem:def en l-grupos}, there is a positive integer $k$ such that $\varphi^* \eq \delta_k^{*}$ in $\mathcal{G}$. Now, for every $\mathbf{A} \in \mathcal{H}_\mathrm{to}$ we have 
\begin{align*}
\mathbf{A} \vDash \varphi & \Leftrightarrow \mathbf{A}^* \vDash \varphi^* \\
& \Leftrightarrow \mathbf{A}^* \vDash \delta_k^* \\
& \Leftrightarrow \mathbf{A} \vDash \delta_k.
\end{align*}
Thus $\varphi \eq \delta_k$ in $\mathcal{H}_\mathrm{to}$ and, as a consequence, also in $\mathcal{H}$.
\end{proof}

\subsection{AE-classes of perfect MV-algebras}

The class of MV-algebras is the equivalent algebraic semantics of \L ukasiewicz infinite-valued logic and has been extensively studied \cite{CigDOtMun00-Book-AlgFoundManyValReasoning}. We recall next its definition and some basic facts.

An {\em MV-algebra} is a structure $\mathbf{A}$ in the language $\tau_\mathcal{MV} := \{+,\neg,0\}$ such that:
\begin{itemize} \setlength{\itemsep}{0pt}\setlength{\parskip}{0pt}
\item $(A,+,0)$ is an Abelian monoid,
\item $\neg \neg x = x$,
\item $x + \neg 0 = \neg 0$,
\item $\neg(\neg x + y) + y = \neg (\neg y + x) + x$.
\end{itemize}
We write $\mathcal{MV}$ for the class of MV-algebras. Given $\mathbf{A} \in \mathcal{MV}$ we define the operations $\vee$ and $\wedge$ by $x \vee y := \neg (\neg x + y) + y$ and $x \wedge y := \neg (\neg x \vee \neg y)$. As is well-known, $(A,\vee,\wedge,0,\neg 0)$ is a bounded distributive lattice whose underlying partial ordering is given by $x \leq y$ if and only if $\neg x + y = \neg 0$. Another relevant derived operation on $\mathbf{A}$ is $*$, which is defined by $x * y := \neg (\neg x + \neg y)$. It is also well-known that $(A,*,\neg 0)$ is an Abelian monoid. We define multiples and powers of $a\in A$ recursively by:
\begin{itemize} \setlength{\itemsep}{0pt}\setlength{\parskip}{0pt}
\item $1a := a$ and $a^1 := a$.
\item $(n+1)a := na + a$ and $a^{n+1} := a^n * a$ for any positive integer $n$.
\end{itemize}

Let $\mathbf{A}$ be an MV-algebra. An {\em ideal} of $\mathbf{A}$ is a nonempty down-set that is closed under $+$. The {\em radical} of $\mathbf{A}$ is the intersection of all maximal ideals of $\mathbf{A}$; we denote it by $\rad \mathbf{A}$. We say that $\mathbf{A}$ is {\em perfect} if it is nontrivial and $A = \rad \mathbf{A} \cup \neg \rad \mathbf{A}$ where $\neg \rad \mathbf{A} := \{\neg a: a \in \rad \mathbf{A}\}$ (this definition is equivalent to the original one given in \cite{BelDiNLet93-LocalMV}, see Corollary 4.5 in that reference). The class of perfect MV-algebras is denoted by $\mathcal{P}$. We shall need the fact that $\rad \mathbf{A} = \{a \in A: a^2 = 0\}$ for $\mathbf{A} \in \mathcal{P}$ (see \cite[Prop. 3.6.4]{CigDOtMun00-Book-AlgFoundManyValReasoning}). We denote the two-element MV-algebra by $\mathbf{2}$; clearly $\mathbf{2} \in \mathcal{P}$. Let $\mathcal{P}_\mathrm{to}$ denote the class of totally ordered perfect MV-algebras.

As in the previous sections, the following two lemmas provide an essential tool for our classification of EFD-sentences.

\begin{lemma} \label{LEMA: props de MV perfectas}
\
\begin{enumerate}
\item The variety $\mathsf{V}(\mathcal{P})$ is arithmetical.
\item An algebra in $\mathsf{V}(\mathcal{P})$ is finitely subdirectly irreducible if and only if it is a totally ordered perfect MV-algebra.
\end{enumerate}
\end{lemma}

\begin{proof}
These results follow from the general theory of MV-algebras. See \cite{CigDOtMun00-Book-AlgFoundManyValReasoning,DiNLet94-PerfectMVCategEquivToAbelLGroups}.
\end{proof}

\begin{lemma} \label{LEMA: equiv en Pto -> equiv en V(P)}
Given EFD-sentences $\varphi, \psi$, if $\varphi \eq \psi$ in $\mathcal{P}_\mathrm{to}$, then $\varphi \eq \psi$ in $\mathsf{V}(\mathcal{P})$.
\end{lemma}

\begin{proof}
The proof is analogous to the one for Lemma \ref{LEMA: equiv en Gto -> equiv en G} using Lemmas \ref{LEMA: aritm + fsi univ} and \ref{LEMA: props de MV perfectas}.
\end{proof}

Next we define a set of EFD-sentences that behave in a special way with respect to the radical. An EFD-sentence $\varphi := \forall \bar{x} \exists ! \bar{z} \, \alpha(\bar{x},\bar{z})$ in the language $\tau_\mathcal{MV}$ belongs to $\Phi_\mathrm{rad}$ if and only if for every $\mathbf{A} \in \mathcal{P}_\mathrm{to}$ and every $\bar{a} \in A^n$ the following holds:
\begin{enumerate}[$(i)$]
\item if $\bar{a} \in (\rad \mathbf{A})^n$ and $\mathbf{A} \vDash \alpha(\bar{a},\bar{b})$ for some $\bar{b} \in A^m$, then $\bar{b} \in (\rad \mathbf{A})^m$,
\item if $\bar{a} \not\in (\rad \mathbf{A})^n$ we have that $\mathbf{A} \vDash \alpha(\bar{a},\bar{0})$ and $\bar{z} = \bar{0}$ is the unique solution to $\alpha(\bar{a},\bar{z})$ in $\mathbf{A}$.
\end{enumerate}

\begin{lemma} \label{LEMA: descomposicion en basicas}
Given an EFD-sentence $\varphi$ in $\tau_\mathcal{MV}$ with a model in $\mathcal{P}$, there are EFD-sentences $\varphi_1,\ldots,\varphi_r \in \Phi_\mathrm{rad}$ such that $\varphi \eq \displaystyle \bigand_{i=1}^r \varphi_i$ in $\mathcal{P}_\mathrm{to}$.
\end{lemma}

\begin{proof}
We introduce some useful notation. Given $\bar{e} := (e_1,\dots,e_n) \in \{0,1\}^n$ and an $n$-tuple of variables $\bar{x} := (x_1,\ldots,x_n)$, define $\bar{x}^{\bar{e}} := (w_1,\dots,w_n)$ where $w_i := x_i$ if $e_i = 0$ and $w_i := \neg x_i$ if $e_i = 1$. We also use this notation for $n$-tuples of elements from an MV-algebra $\mathbf{A}$: if $\bar{a} \in A^n$, then $\bar{a}^{\bar{e}} := \bar{b}$ where $b_i := a_i$ if $e_i = 0$ and $b_i := \neg a_i$ if $e_i = 1$. Finally define
\begin{align*}
\rho(\bar{x}) & := x_1^2 = 0 \band \dots \band x_n^2 = 0, \\
\widetilde{\rho}(\bar{x}) & := (\neg x_1 \wedge \dots \wedge \neg x_n)^2 = 0.
\end{align*}
Note that given $\mathbf{A} \in \mathcal{P}$ and $\bar{a} \in A^n$ we have $\mathbf{A} \vDash \rho(\bar{a})$ if and only if each $a_i \in \rad \mathbf{A}$, and $\widetilde{\rho}(\bar{x})$ is equivalent over $\mathcal{P}$ to the negation of $\rho(\bar{x})$.

Assume $\varphi := \forall \bar{x} \exists! \bar{z} \, \alpha(\bar{x},\bar{z})$ has a model in $\mathcal{P}$. Since $\mathbf{2}$ is both a subalgebra and a quotient of any member of $\mathcal{P}$, we have that $\mathbf{2} \vDash \varphi$. Thus, given $\bar{e} \in \{0,1\}^n$, there is a unique $\bar{e}' \in \{0,1\}^m$ such that $\alpha(\bar{e},\bar{e}')$ holds.
Let $\eta_{\bar{e}}(\bar{x},\bar{z})$ be the formula 
$$(\rho(\bar{x}) \band \alpha(\bar{x}^{\bar{e}},\bar{z}^{\bar{e}'})) \mathrel{\mathrm{or}} (\widetilde{\rho}(\bar{x}) \band \bar{z} = \bar{0}).$$
Now observe that on any totally ordered (perfect) MV-algebra the disjunction $$x = y \mathrel{\mathrm{or}} z = w$$ is equivalent to the equation $$((x * \neg y) + (y * \neg x)) \wedge ((z * \neg w) + (w * \neg z)) = 0.$$ Thus there is a conjunction of equations $\alpha_{\bar{e}}(\bar{x},\bar{z})$ equivalent to $\eta_{\bar{e}}(\bar{x},\bar{z})$ in $\mathcal{P}_\mathrm{to}$. Finally define $$\varphi_{\bar{e}} := \forall \bar{x} \exists! \bar{z} \, \alpha_{\bar{e}}(\bar{x},\bar{z}).$$
We claim that $\varphi$ is equivalent in $\mathcal{P}_\mathrm{to}$ to the conjunction of all $\varphi_{\bar{e}}$ for $\bar{e} \in \{0,1\}^n$, and that each of these formulas belongs to $\Phi_\mathrm{rad}$.

Fix $\mathbf{A} \in \mathcal{P}_\mathrm{to}$ and assume $\mathbf{A} \vDash \varphi$. Let $\bar{e} \in \{0,1\}^n$ and $\bar{a} \in A^n$. We show that there is $\bar{b} \in A^m$ such that $\mathbf{A} \vDash \alpha_{\bar{e}}(\bar{a},\bar{b})$. If $\mathbf{A} \vDash \rho(\bar{a})$, take $\bar{c}$ as the unique $m$-tuple such that $\mathbf{A} \vDash \alpha(\bar{a}^{\bar{e}},\bar{c})$. Then $\mathbf{A} \vDash \rho(\bar{a}) \band \alpha(\bar{a}^{\bar{e}},(\bar{c}^{\bar{e}'})^{\bar{e}'})$ since $(\bar{c}^{\bar{e}'})^{\bar{e}'} = \bar{c}$, so we can take $\bar{b} := \bar{c}^{\bar{e}'}$. If $\mathbf{A} \not\vDash \rho(\bar{a})$, take $\bar{b} := \bar{0}$. It is easy to see that $\bar{b}$ is unique and so we have that $\mathbf{A} \vDash \varphi_{\bar{e}}$. 

Conversely, assume $\mathbf{A} \vDash \varphi_{\bar{e}}$ for every $\bar{e} \in \{0,1\}^n$ and fix $\bar{a} \in A^n$. There is $\bar{e} \in \{0,1\}^n$ such that $\bar{a}^{\bar{e}} \in (\rad \mathbf{A})^n$, that is, such that $\mathbf{A} \vDash \rho(\bar{a}^{\bar{e}})$. Since $\mathbf{A} \vDash \varphi_{\bar{e}}$, there is $\bar{b} \in A^m$ such that $\mathbf{A} \vDash \alpha((\bar{a}^{\bar{e}})^{\bar{e}},\bar{b}^{\bar{e}'})$, that is, $\mathbf{A} \vDash \alpha(\bar{a},\bar{b}^{\bar{e}'})$. This shows that $\mathbf{A} \vDash E(\varphi)$. To check that $\mathbf{A} \vDash U(\varphi)$, suppose $\bar{c} \in A^m$ is such that $\mathbf{A} \vDash \alpha(\bar{a},\bar{c})$. Since $\mathbf{A} \vDash U(\varphi_{\bar{e}})$ and $\mathbf{A} \vDash \alpha((\bar{a}^{\bar{e}})^{\bar{e}},(\bar{c}^{\bar{e}'})^{\bar{e}'})$, we have $\bar{b} = \bar{c}^{\bar{e}'}$, and thus $\bar{c} = \bar{b}^{\bar{e}'}$.

It remains to show that every $\varphi_{\bar{e}} \in \Phi_\mathrm{rad}$. Let $\mathbf{B} \in \mathcal{P}_\mathrm{to}$ and take $\bar{a} \in (\rad \mathbf{B})^n$, $\bar{b} \in B^m$ such that $\mathbf{B} \vDash \alpha_{\bar{e}}(\bar{a},\bar{b})$. Note that $\mathbf{B} \vDash \alpha(\bar{a}^{\bar{e}},\bar{b}^{\bar{e}'})$. If $h\colon \mathbf{B} \to \mathbf{2}$ is the homomorphism with kernel $\rad \mathbf{B}$, it follows that $\mathbf{2} \vDash \alpha(\bar{e},h(\bar{b})^{\bar{e}'})$. So, as $\mathbf{2} \vDash U(\varphi)$, we obtain $h(\bar{b})^{\bar{e}'} = \bar{e}'$. Hence $h(\bar{b}) = \bar{0}$, that is, $\bar{b} \in (\rad \mathbf{B})^n$. Condition $(ii)$ in the definition of $\Phi_\mathrm{rad}$ holds by construction of $\alpha_{\bar{e}}(\bar{x},\bar{z})$.
\end{proof}

Next we define a family of EFD-sentences that serve the same purpose as the $\delta_k$'s in our previous section. For each positive integer $k$ define the $\tau_\mathcal{MV}$-term
$$\mathsf{t}_k(z) := (kz \wedge \neg 2z^2) \vee z^k$$ and the EFD-sentence $$\varepsilon_k := \forall x \exists! z \, \mathsf{t}_k(z) = x.$$
In the following lemma we collect several properties of these terms and EFD-sentences.

\begin{lemma} \label{LEMA: props de t_k y delta_k}
Let $\mathbf{A} \in \mathcal{P}$.
\begin{enumerate}
\item For every $a \in A$ we have $$\mathsf{t}_k^\mathbf{A}(a) = \begin{cases} a^k & \text{ if } a \in \neg \rad \mathbf{A}, \\ ka & \text{ if } a \in \rad \mathbf{A}. \end{cases}$$
\item The term-function $\mathsf{t}_k^\mathbf{A}$ is a one-to-one endomorphism of $\mathbf{A}$.
\item The following are equivalent:
\begin{enumerate}[$(i)$]
\item $\mathsf{t}_k^\mathbf{A}$ is surjective,
\item $\mathbf{A} \vDash \varepsilon_k$,
\item For every $a \in \rad \mathbf{A}$, there is $b \in \rad \mathbf{A}$ such that $kb = a$.
\item For every $a \in \neg \rad \mathbf{A}$, there is $b \in \neg \rad \mathbf{A}$ such that $b^k = a$.
\end{enumerate}
\item If $\mathbf{A} \vDash \varepsilon_k$, then $[\varepsilon_k]^\mathbf{A}$ is an automorphism, which is the inverse of $\mathsf{t}_k^\mathbf{A}$.
\end{enumerate}
\end{lemma}

\begin{proof}
Item 1. follows directly from the fact that $2a^2 = 0$ for $a \in \rad \mathbf{A}$ and $2a^2 = 1$ for $a \in \neg \rad \mathbf{A}$.

%
%
To prove 2. we show first that $\mathsf{t}_k^\mathbf{A}$ preserves $+$, that is, $\mathsf{t}_k^\mathbf{A}(a+b) = \mathsf{t}_k^\mathbf{A}(a) + \mathsf{t}_k^\mathbf{A}(b)$. We consider three different cases and use item 1. in each case. If $a,b \in \rad \mathbf{A}$, condition $\mathsf{t}_k^\mathbf{A}(a+b) = \mathsf{t}_k^\mathbf{A}(a) + \mathsf{t}_k^\mathbf{A}(b)$ reduces to $k(a+b) = ka + kb$, which holds in any MV-algebra. If $a,b \in \neg \rad \mathbf{A}$, we prove that $(a+b)^k = a^k + b^k$. Indeed, this equation holds since $x+y = 1$ for every $x,y \in \neg \rad \mathbf{A}$. Finally, if $a \in \rad \mathbf{A}$ and $b \in \neg \rad \mathbf{A}$, we show that $(a+b)^k = ka + b^k$. Indeed, in any MV-algebra $k(\neg b \vee a) = k(\neg b) \vee ka$. Now $k(\neg b \vee a) = k(\neg(\neg \neg b + a)+a) = k(\neg(a+b)+a) = k \neg (a+b) + ka$ and $k(\neg b) \vee ka = \neg(\neg k \neg b + ka) + ka = \neg(ka + b^k) + ka$. This shows that $k \neg (a+b) + ka = \neg(ka + b^k) + ka$ and, since $(\rad \mathbf{A},+)$ is a cancellative semigroup (\cite[Lemma 3.2]{DiNLet94-PerfectMVCategEquivToAbelLGroups}), we conclude that $k \neg (a+b) = \neg (ka+b^k)$, so $(a+b)^k = ka+b^k$ as was to be proved. The fact that $\mathsf{t}_k^\mathbf{A}$ preserves $0$ and $\neg$ is straightforward. Thus, $\mathsf{t}_k^\mathbf{A}$ is an endomorphism of $\mathbf{A}$. To show that it is one-to-one, it is enough to prove that $t_k^\mathbf{A}(a) = 0$ implies $a = 0$. Indeed, if $a \in \rad \mathbf{A}$, then $\mathsf{t}_k^\mathbf{A}(a) = ka = 0$, so $a \leq ka = 0$; if $a \in \neg \rad \mathbf{A}$, then $\mathsf{t}_k^\mathbf{A}(a) = a^k \in \neg \rad \mathbf{A}$ and cannot equal $0$.

Item 3. is a direct consequence of 1. and 2., and 4. follows easily from 2. and 3.
\end{proof}

Mundici's functor $\Gamma$ \cite{CigDOtMun00-Book-AlgFoundManyValReasoning} allows us to  make explicit the connection between $\delta_k$, defined in the previous section, and $\varepsilon_k$.  Given $\mathbf{A} \in \mathcal{P}$, there is an Abelian $\ell$-group $\mathbf{G}$ such that $\mathbf{A} \cong \Gamma(\mathbf{Z} \mathbin{\vec{\times}} \mathbf{G}, (1,0))$ where $\mathbf{Z}$ is the $\ell$-groups of integers and $\vec{\times}$ is the lexicographic product (note that $\mathbf{G}^+ \cong \radn \mathbf{A}$). Now, Lemma \ref{LEMA: props de t_k y delta_k}.3.$(iii)$ says that $\mathbf{A} \vDash \varepsilon_k$ if and only if $\mathbf{G} \vDash \delta_k$.

Given a model $\mathbf{A}$ of $\varepsilon_k$ we write $\mathsf{d}_k$ for the function $[\varepsilon_k]^\mathbf{A}$. In view of Lemma \ref{LEMA: props de t_k y delta_k}.1. we see that $\mathsf{d}_k$ is the analogue of division by $k$ in $\ell$-groups. To illustrate the behaviour of these functions we look at a concrete case.

\begin{example}\label{EJ: Gamma(Z por Q)}
Let $\mathbf{Q}$ be the $\ell$-group of rational numbers and consider the perfect MV-algebra $\mathbf{D} := \Gamma(\mathbf{Z} \mathbin{\vec{\times}} \mathbf{Q}, (1,0))$. Recall that the universe of $\mathbf{D}$ is $\{(0,x): x \in Q, x \geq 0\} \cup \{(1,x): x \in Q, x \leq 0\}$. It is straightforward to check that $\mathbf{D} \vDash \varepsilon_k$ for every positive integer $k$ and that $\mathsf{t}_k^{\mathbf{D}}(i,x) = (i,kx)$ and $\mathsf{d}^{\mathbf{D}}_k(i,x) = (i,\frac{x}{k})$ for every $(i,x)$ from $\mathbf{D}$.
\end{example}

Next we show how the results for cancellative hoops translate to perfect MV-algebras. 
Given $\mathbf{A} \in \mathcal{MV}$, we define $$x \dotdiv y := \neg(\neg x + y).$$
Via this shorthand we can interpret $\tau_\mathcal{H}$-terms in MV-algebras. The radical of $\mathbf{A}$ is closed under $+$ and $\dotdiv$. Moreover, $\radn \mathbf{A} := (\rad \mathbf{A},+,\dotdiv,0)$ is a cancellative hoop (see \cite[Lemma 3.2]{DiNLet94-PerfectMVCategEquivToAbelLGroups}). Note that with these definitions it is obvious that for an $\tau_\mathcal{H}$-term $t(\bar{x})$ and $\bar{a} \in (\rad \mathbf{A})^n$ we have $t^\mathbf{A}(\bar{a}) = t^{\radn \mathbf{A}}(\bar{a})$.

Recall from Section \ref{DEF: delta_k} that $\delta_k := \forall x \exists! z \, kz = x$. 

\begin{lemma} \label{LEMA: A sat phi^p <-> Rad A sat phi}
For every positive integer $k$ and every $\mathbf{A} \in \mathcal{P}$ we have $$\mathbf{A} \vDash \varepsilon_k\Leftrightarrow \radn \mathbf{A} \vDash \delta_k.$$
\end{lemma}

\begin{proof}
The equivalence of $(ii)$ and $(iii)$ in Lemma \ref{LEMA: props de t_k y delta_k}.3. shows that $\mathbf{A} \vDash E(\varepsilon_k) \Leftrightarrow \radn \mathbf{A} \vDash E(\delta_k)$. Since $U(\varepsilon_k)$ is valid in $\mathcal{P}$ and $U(\delta_k)$ is valid in $\mathcal{H}$, the lemma now follows.
\end{proof}

\begin{lemma} \label{LEMA: A sat phi <-> Rad A sat phi^h}
For each $\varphi \in \Phi_\mathrm{rad}$ with a model in $\mathcal{P}$, there is an EFD-sentence $\varphi^\mathrm{H}$ in the language $\tau_\mathcal{H}$ such that for every $\mathbf{A} \in \mathcal{P}_\mathrm{to}$ $$\mathbf{A} \vDash \varphi \Leftrightarrow \radn \mathbf{A} \vDash \varphi^\mathrm{H}.$$
\end{lemma}

\begin{proof}
Let $\varphi := \forall \bar{x} \exists ! \bar{z} \, \alpha(\bar{x},\bar{z})$ in $\Phi_\mathrm{rad}$. We can assume $\alpha(\bar{x},\bar{z})$ is a conjunction of equations of the form $t(\bar{x},\bar{z}) =  0$, where $t$ is an $\tau_\mathcal{MV}$-term. As in the proofs of the previous lemmas, since $\varphi$ has a model in $\mathcal{P}$, we know that $\mathbf{2} \vDash \varphi$. Moreover, since $\varphi \in \Phi_\mathrm{rad}$, we have $\mathbf{2} \vDash t(\bar{0},\bar{0}) = 0$. Thus, by \cite[Theorem 3.1]{AglPan02-GeomMethWajsHoops}, there is an $\tau_\mathcal{H}$-term $t'(\bar{x},\bar{z})$ such that $\mathcal{MV} \vDash \forall \bar{x} \bar{z} \, t(\bar{x},\bar{z}) = t'(\bar{x},\bar{z})$. Let $\beta(\bar{x},\bar{z})$ be the result of replacing each $\tau_\mathcal{MV}$-term in $\alpha(\bar{x},\bar{z})$ by an equivalent $\tau_\mathcal{H}$-term; define $\varphi^\mathrm{H} := \forall \bar{x} \exists! \bar{z} \, \beta(\bar{x},\bar{z})$.

Fix $\mathbf{A} \in \mathcal{P}_\mathrm{to}$. Suppose $\mathbf{A} \vDash \varphi$ and take $\bar{a} \in (\rad \mathbf{A})^n$. Since $\varphi \in \Phi_\mathrm{rad}$, there is $\bar{b} \in (\rad \mathbf{A})^m$ such that $\mathbf{A} \vDash \alpha(\bar{a},\bar{b})$, and thus $\radn \mathbf{A} \vDash \beta(\bar{a},\bar{b})$. Furthermore, if $\bar{c} \in (\rad \mathbf{A})^m$ is such that $\radn \mathbf{A} \vDash \beta(\bar{a},\bar{c})$, then $\mathbf{A} \vDash \alpha(\bar{a},\bar{c})$, and it follows that $\bar{c} = \bar{b}$. This completes the proof of $\radn \mathbf{A} \vDash \varphi^\mathrm{H}$.

For the other direction assume $\radn \mathbf{A} \vDash \varphi^\mathrm{H}$ and let $\bar{a} \in A^n$. If $\bar{a} \notin (\rad \mathbf{A})^n$, the definition of $\Phi_\mathrm{rad}$ implies that there is a unique $\bar{b}$ such that $\mathbf{A} \vDash \alpha(\bar{a},\bar{b})$, namely $\bar{b} = \bar{0}$. To conclude, suppose $\bar{a} \in (\rad \mathbf{A})^n$. Since $\radn \mathbf{A} \vDash \varphi^\mathrm{H}$, there is $\bar{b} \in (\rad \mathbf{A})^m$ such that $\radn \mathbf{A} \vDash \beta(\bar{a},\bar{b})$, and thus $\mathbf{A} \vDash \alpha(\bar{a},\bar{b})$. If $\bar{c} \in A^m$ is such that $\mathbf{A} \vDash \alpha(\bar{a},\bar{c})$, then as $\varphi \in \Phi_\mathrm{rad}$ we know that $\bar{c} \in (\rad \mathbf{A})^m$. Since $\radn \mathbf{A} \vDash U(\varphi^\mathrm{H})$, it follows that $\bar{c} = \bar{b}$.
\end{proof}

Recall that the identity $\forall x \, 2x = x$ axiomatizes the class of Boolean algebras relative to the class of MV-algebras. Thus the only model of this identity in $\mathcal{P}$ is the two-element MV-algebra.

\begin{lemma} \label{LEMA: basicas equiv delta_k o boole}
Given $\varphi \in \Phi_\mathrm{rad}$ with a model in $\mathcal{P}$ either $\varphi \eq \forall x \, 2x = x$ in $\mathcal{P}_\mathrm{to}$ or there is a positive integer $k$ such that $\varphi \eq \varepsilon_k$ in $\mathcal{P}_\mathrm{to}$.
\end{lemma}

\begin{proof}
If $\mathbf{2}$ is the only model of $\varphi$ in $\mathcal{P}_\mathrm{to}$, then $\varphi \eq \forall x \, 2x = x$ in $\mathcal{P}_\mathrm{to}$. Assume $\varphi$ has a model $\mathbf{A} \in \mathcal{P}_\mathrm{to}$ non-isomorphic with $\mathbf{2}$. Then $\varphi^\mathrm{H}$ has $\radn \mathbf{A}$ as a nontrivial model (see Lemma \ref{LEMA: A sat phi <-> Rad A sat phi^h}). By Theorem \ref{TEO: phi <-> phi_k in H}, there is a positive integer $k$ such that $\varphi^\mathrm{H} \eq \delta_k$ in $\mathcal{H}$. From these facts and Lemma \ref{LEMA: A sat phi^p <-> Rad A sat phi}, for every $\mathbf{B} \in \mathcal{P}_\mathrm{to}$ we have
\begin{align*}
\mathbf{B} \vDash \varphi & \Leftrightarrow \radn \mathbf{B} \vDash \varphi^\mathrm{H} \\
& \Leftrightarrow \radn \mathbf{B} \vDash \delta_k \\
& \Leftrightarrow \mathbf{B} \vDash \varepsilon_k.
\end{align*}
\end{proof}

We are now in the position to prove a characterization of EFD-sentences for the variety $\mathsf{V}(\mathcal{P})$.

\begin{theorem} \label{EFDs en perfectas}
For every EFD-sentence $\varphi$ in $\tau_\mathcal{MV}$ with a  model in $\mathcal{P}$ either $\varphi \eq \forall x \, 2x = x$ in $\mathsf{V}(\mathcal{P})$ or there is a positive integer $k$ such that $\varphi \eq \varepsilon_k$ in $\mathsf{V}(\mathcal{P})$.
\end{theorem}

\begin{proof}
By Lemma \ref{LEMA: descomposicion en basicas}, there are basic EFD-sentences $\varphi_1, \dots, \varphi_r$ such that $\varphi \eq \displaystyle\bigand_{i=1}^r \varphi_i$ in $\mathcal{P}_\mathrm{to}$.

First suppose that $\varphi$ has a model in $\mathcal{P}$ non-isomorphic with $\mathbf{2}$; then so does each $\varphi_i$. By Lemma \ref{LEMA: basicas equiv delta_k o boole}, there are positive integers $k_1, \ldots, k_r$ such that $\varphi_i \eq \varepsilon_{k_i}$ in $\mathcal{P}_\mathrm{to}$ for every $i \in \{1,\ldots,r\}$. Thus $\varphi \eq \displaystyle\bigand_{i=1}^r \varepsilon_{k_i}$ in $\mathcal{P}_\mathrm{to}$. Now take $k := k_1\ldots k_r$ and note that $\displaystyle \bigand_{i=1}^r \varepsilon_{k_i} \eq \varepsilon_k$ in $\mathcal{P}_\mathrm{to}$. Hence $\varphi \eq \varepsilon_k$ in $\mathcal{P}_\mathrm{to}$. Finally, by Lemma \ref{LEMA: equiv en Pto -> equiv en V(P)}, we get that $\varphi \eq \varepsilon_k$ in $\mathsf{V}(\mathcal{P})$.

Now, if $\mathbf{2}$ is the only model of $\varphi$ in $\mathcal{P}$, then $\varphi \eq \forall x \, 2x = x$ in $\mathcal{P}_\mathrm{to}$. Note that $\forall x \, 2x = x$ is equivalent to the EFD-sentence $\forall x \exists! z \, (2x = x) \band (z = x)$, so we can apply Lemma \ref{LEMA: equiv en Pto -> equiv en V(P)} to conclude that $\varphi \eq \forall x \, 2x = x$ in $\mathsf{V}(\mathcal{P})$.
\end{proof}

As in the case of $\ell$-groups, the characterization of EFD-sentences easily provides a description of the AE-classes. Given a set $S$ of prime numbers, let $\Sigma_S := \{\varepsilon_p: p \in S\}$. 

\begin{theorem} \label{TEO: clases AE de perfectas}
Every set of EFD-sentences in $\tau_\mathcal{MV}$ is equivalent over $\mathsf{V}(\mathcal{P})$ to exactly one of the following:
\begin{itemize}
\item $\{\forall xy \, x = y\}$,
\item $\{\forall x \, 2x = x\}$,
\item $\Sigma_S$ for some set $S$ of prime numbers. \end{itemize}
Furthermore, the map $S \mapsto \Sigma_S$ is one-to-one, and thus, the lattice of AE-subclasses of $\mathsf{V}(\mathcal{P})$ is isomorphic with $\mathbf{2} \oplus \mathbf{2}^\omega$.
\end{theorem}

\begin{proof}
From Lemma \ref{LEMA: props de t_k y delta_k}.3. it is easy to see that $\varepsilon_k$ is equivalent over $\mathsf{V}(\mathcal{P})$ to $\{\varepsilon_p: p \text{ prime divisor of } k\}$. This fact together with Theorem \ref{EFDs en perfectas} proves the first part the the theorem.

Now, given a set $S$ of positive primes, if we consider the $\ell$-group $\mathbf{G}$ of rational numbers whose denominators are products of primes in $S$, then for every prime $p$ we have that the algebra $\Gamma(\mathbf{Z} \mathbin{\vec{\times}} \mathbf{G}, (1,0))$ satisfies $\varepsilon_p$ if and only if $p \in S$. Finally, observe that Boolean algebras trivially satisfy $\varepsilon_p$ for every prime number $p$. This proves the furthermore part.
\end{proof}

\subsection{The algebraic expansions of $L_\mathcal{P}$}

The Logic $L_\mathcal{P}$ of Perfect MV-Algebras \cite{DiNLet94-PerfectMVCategEquivToAbelLGroups} is the extension of \L ukasiewicz Logic by the axiom $2x^2 \mathbin{\leftrightarrow} (2x)^2$ (recall that $x \mathbin{\leftrightarrow} y := (\neg x + y) \wedge (\neg y + x)$). As the name suggests, the equivalent algebraic semantics of $L_\mathcal{P}$ is the variety $\mathsf{V}(\mathcal{P})$.

Given a prime number $p$, the algebraic expansion of $L_\mathcal{P}$ corresponding to the EFD-sentence $\varepsilon_p$ is, by definition, obtained from $L_\mathcal{P}$ by adding the axiom:
\begin{equation}
((p\mathsf{d}_p(x) \wedge \neg 2\mathsf{d}_p(x)^2) \vee \mathsf{d}_p(x)^p) \mathbin{\leftrightarrow} x \tag{$D_p$},
\end{equation}
and the rule $\mathrm{U}_{\{((kz \wedge \neg 2z^2) \vee z^k) \leftrightarrow x\}}$. Since this rule is derivable in $L_\mathcal{P}$, the expansion is obtained simply by adding $D_p$.

For a set $S$ of prime numbers define $L_\mathcal{P}^S$ as the expansion of $L_\mathcal{P}$ by the axioms $\{D_p: p \in S\}$. Note that, by the comment above, $L_\mathcal{P}^S$ is the algebraic expansion of $L_\mathcal{P}$ corresponding to the AE-class axiomatized by $\Sigma_S := \{\varepsilon_p: p \in S\}$. Thus, the equivalent algebraic semantics $\mathsf{V}(\mathcal{P})^{\Sigma_S}$ of $L_\mathcal{P}^S$ is a variety.

\begin{theorem}
\
\begin{enumerate}
\item Every algebraic expansion of $L_\mathcal{P}$ is $\tau_\mathcal{MV}$-equipollent to exactly one of the following:
\begin{itemize}
\item Inconsistent Logic,
\item Classical Propositional Logic,
\item $L_\mathcal{P}^S$ for some set $S$ of prime numbers.
\end{itemize}
\item The algebraic expansions of $L_\mathcal{P}$ form a lattice isomorphic with $\mathbf{2}^\omega \oplus \mathbf{2}$ when ordered by $\tau_\mathcal{MV}$-morphisms.
\item Given sets $S,S'$ of prime numbers with $S \subseteq S'$, the expansion $L_\mathcal{P}^{S'}$ is conservative over $L_\mathcal{P}^S$.
\end{enumerate}
\end{theorem}

\begin{proof}
Items 1. and 2. follow from Theorems \ref{TEO: equivalencias interpretaciones} and \ref{TEO: clases AE de perfectas}. To prove 3. let $\mathbf{D}_S$ be the expansion of $\Gamma(\mathbf{Z} \mathbin{\vec{\times}} \mathbf{Q}, (1,0))$ by the operations $\mathsf{d}_p$ for $p \in S$. Now the proof follows the argument of that of 3. in Theorem \ref{TEO: exp alg de Bal} with $\mathbf{D}_S$ in place of $\mathbf{Q}_S$. 
\end{proof}

\bibliographystyle{plain}

\begin{thebibliography}{10}

\bibitem{AglPan02-GeomMethWajsHoops}
P.~Aglian\`o and G.~Panti.
\newblock Geometrical methods in {W}ajsberg hoops.
\newblock {\em J. Algebra}, 256(2):352--374, 2002.

\bibitem{BelDiNLet93-LocalMV}
L.~P. Belluce, A.~Di~Nola, and A.~Lettieri.
\newblock Local {MV}-algebras.
\newblock {\em Rend. Circ. Mat. Palermo (2)}, 42(3):347--361, 1993.

\bibitem{BelDiNGer07-PerfectMVAlg}
Lawrence~P. Belluce, Antonio Di~Nola, and Brunella Gerla.
\newblock Perfect {MV}-algebras and their logic.
\newblock {\em Appl. Categ. Structures}, 15(1-2):135--151, 2007.

\bibitem{Birkhoff40-Book-LatticeTheory}
Garrett Birkhoff.
\newblock {\em Lattice {T}heory}.
\newblock American Mathematical Society, New York, 1940.

\bibitem{BloFer00-OnStructHoops}
W.~J. Blok and I.~M.~A. Ferreirim.
\newblock On the structure of hoops.
\newblock {\em Algebra Universalis}, 43(2-3):233--257, 2000.

\bibitem{Bosbach69-KomplHalbgruppen-AxiomArith}
Bruno Bosbach.
\newblock Komplement\"are {H}albgruppen: {A}xiomatik und {A}rithmetik.
\newblock {\em Fund. Math.}, 64:257--287, 1969.

\bibitem{Bosbach70-KomplHalbgruppen-KongQuot}
Bruno Bosbach.
\newblock Komplement\"are {H}albgruppen. {K}ongruenzen und {Q}uotienten.
\newblock {\em Fund. Math.}, 69:1--14, 1970.

\bibitem{BurSan81-Book-CourseUnivAlg}
S.~Burris and H.~Sankappanavar.
\newblock {\em A course in universal algebra}.
\newblock Graduate texts in mathematics. Springer-Verlag, 1981.

\bibitem{Caicedo04-ImpConnectivesAlgLogics}
Xavier Caicedo.
\newblock Implicit connectives of algebraizable logics.
\newblock {\em Studia Logica}, 78(1-2):155--170, 2004.

\bibitem{Caicedo07-Implicit-in-Luka}
Xavier Caicedo.
\newblock {\em Algebraic and Proof-theoretic Aspects of Non-classical Logics,
  Lecture Notes in Computer Science}, volume 4460, chapter Implicit Operations
  in MV-Algebras and the Connectives of \L ukasiewicz Logic.
\newblock Springer, Berlin, Heidelberg, 2007.

\bibitem{CaiCig01-AlgApprIntuitConnectives}
Xavier Caicedo and Roberto Cignoli.
\newblock An algebraic approach to intuitionistic connectives.
\newblock {\em J. Symbolic Logic}, 66(4):1620--1636, 2001.

\bibitem{CalGon05-EquipollentLogSys}
Carlos Caleiro and Ricardo Gon\c{c}alves.
\newblock Equipollent logical systems.
\newblock In {\em Logica universalis}, pages 99--111. Birkh\"{a}user, Basel,
  2005.

\bibitem{CamVag11-AlgebraicFunctions}
M.~Campercholi and D.~Vaggione.
\newblock Algebraic functions.
\newblock {\em Studia Logica}, 98(1-2):285--306, 2011.

\bibitem{Campercholi10-Heyting-preprint}
Miguel Campercholi.
\newblock Algebraically expandable classes of heyting algebras.
\newblock 2010.
\newblock preprint.

\bibitem{Campercholi10-Implication}
Miguel Campercholi.
\newblock Algebraically expandable classes of implication algebras.
\newblock {\em Internat. J. Algebra Comput.}, 20(5):605--617, 2010.

\bibitem{CamVag09-AlgExpClasses}
Miguel Campercholi and Diego Vaggione.
\newblock Algebraically expandable classes.
\newblock {\em Algebra Universalis}, 61(2):151--186, 2009.

\bibitem{CigDOtMun00-Book-AlgFoundManyValReasoning}
Roberto L.~O. Cignoli, Itala M.~L. D'Ottaviano, and Daniele Mundici.
\newblock {\em Algebraic foundations of many-valued reasoning}, volume~7 of
  {\em Trends in Logic---Studia Logica Library}.
\newblock Kluwer Academic Publishers, Dordrecht, 2000.

\bibitem{DiNLet94-PerfectMVCategEquivToAbelLGroups}
Antonio Di~Nola and Ada Lettieri.
\newblock Perfect {MV}-algebras are categorically equivalent to abelian
  {$l$}-groups.
\newblock {\em Studia Logica}, 53(3):417--432, 1994.

\bibitem{Ferreirim92-Thesis}
Isabel Maria~Andre Ferreirim.
\newblock {\em On varieties and quasivarieties of hoops and their reducts}.
\newblock ProQuest LLC, Ann Arbor, MI, 1992.
\newblock Thesis (Ph.D.)--University of Illinois at Chicago.

\bibitem{Font16-Book-AbstractAlgebraicLogic}
Josep~Maria Font.
\newblock {\em Abstract algebraic logic}, volume~60 of {\em Studies in Logic
  (London)}.
\newblock College Publications, London, 2016.
\newblock An introductory textbook, Mathematical Logic and Foundations.

\bibitem{GLS04}
Adriana Galli, Renato~A. Lewin, and Marta Sagastume.
\newblock The logic of equilibrium and abelian lattice ordered groups.
\newblock {\em Arch. Math. Logic}, 43(2):141--158, 2004.

\bibitem{Glass99-Book-PartiallyOrderedGroups}
A.~M.~W. Glass.
\newblock {\em Partially ordered groups}, volume~7 of {\em Series in Algebra}.
\newblock World Scientific Publishing Co., Inc., River Edge, NJ, 1999.

\bibitem{GraVag96-SheafRepVarLatExp}
Hector Gramaglia and Diego Vaggione.
\newblock Birkhoff-like sheaf representation for varieties of lattice
  expansions.
\newblock {\em Studia Logica}, 56(1-2):111--131, 1996.
\newblock Special issue on Priestley duality.

\bibitem{Hodges93-Book-ModelTheory}
Wilfrid Hodges.
\newblock {\em Model theory}, volume~42 of {\em Encyclopedia of Mathematics and
  its Applications}.
\newblock Cambridge University Press, Cambridge, 1993.

\bibitem{Jonsson67-AlgCongLatDistributive}
Bjarni J\'onsson.
\newblock Algebras whose congruence lattices are distributive.
\newblock {\em Math. Scand.}, 21:110--121 (1968), 1967.

\bibitem{MOG05}
George Metcalfe, Nicola Olivetti, and Dov Gabbay.
\newblock Sequent and hypersequent calculi for abelian and \l ukasiewicz
  logics.
\newblock {\em ACM Trans. Comput. Log.}, 6(3):578--613, 2005.

\bibitem{Mundici86-InterpAFinLuk}
Daniele Mundici.
\newblock Interpretation of {AF} {$C^\ast$}-algebras in \l ukasiewicz
  sentential calculus.
\newblock {\em J. Funct. Anal.}, 65(1):15--63, 1986.

\bibitem{Mundici86-MappAbelLGroupsStrongUnitMVAlg}
Daniele Mundici.
\newblock Mapping abelian {$l$}-groups with strong unit one-one into {MV}
  algebras.
\newblock {\em J. Algebra}, 98(1):76--81, 1986.

\bibitem{Robinson56-Book-CompleteTheories}
Abraham Robinson.
\newblock {\em Complete theories}.
\newblock North-Holland Publishing Co., Amsterdam, 1956.

\bibitem{Volger79-PreservThmLimitsStrucGlobalSections}
Hugo Volger.
\newblock Preservation theorems for limits of structures and global sections of
  sheaves of structures.
\newblock {\em Math. Z.}, 166(1):27--54, 1979.

\bibitem{Weinberg65-FreeLatOrdAbelGroupsII}
Elliot~Carl Weinberg.
\newblock Free lattice-ordered abelian groups. {II}.
\newblock {\em Math. Ann.}, 159:217--222, 1965.

\end{thebibliography}

\

\

\

M. Campercholi

\noindent Facultad de Matem\'atica, Astronom\'{i}a y F\'{i}sica (Universidad Nacional de C\'ordoba) \\
CIEM - CONICET \\
C\'ordoba, Argentina

\noindent camper@famaf.unc.edu.ar

\

D. N. Casta\~no

\noindent Departamento de Matem\'atica (Universidad Nacional del Sur) \\
Insituto de Matemática (INMABB) - UNS-CONICET \\
Bah\'{i}a Blanca, Argentina

\noindent diego.castano@uns.edu.ar

\

J. P. D\'{i}az Varela

\noindent Departamento de Matem\'atica (Universidad Nacional del Sur) \\
Insituto de Matemática (INMABB) - UNS-CONICET \\
Bah\'{i}a Blanca, Argentina

\noindent usdiavar@criba.edu.ar

\

J. Gispert

\noindent Departament de Matem\`atiques i Inform\`atica \\
Institut de Matem\`atiques de la Universitat de Barcelona (IMUB) \\
Barcelona Graduate School of Mathematics (BGSMath) \\
Universitat de Barcelona (UB) \\
Barcelona, Spain.

\noindent jgispertb@ub.edu

\end{document}